\documentclass[12pt]{amsart}
\usepackage{amssymb}
\usepackage[polish,english]{babel}
\usepackage{lmodern}
\usepackage[utf8]{inputenc}
\usepackage[T1]{fontenc}
\usepackage{amsmath,amscd}
\usepackage[normalem]{ulem}
\usepackage{listings}
\usepackage{amsaddr}

\usepackage{xcolor}
\usepackage{graphicx}

\usepackage[section]{placeins}

\usepackage[hidelinks]{hyperref}
\newcommand{\doi}[1]{\\ \href{https://doi.org/#1}{\texttt{doi:#1}}}
\usepackage{url}

\usepackage{underscore}

\newtheorem{theorem}{Theorem}

\newtheorem{algorithm}[theorem]{Algorithm}

\newtheorem{definition}[theorem]{Definition}

\newcommand{\kw}[1]{\textbf{\textit{#1}}}

\newcommand{\rec}{\operatorname{rec}}
\newcommand{\diam}{\operatorname{diam}}
\newcommand{\dist}{\operatorname{dist}}

\newcommand{\card}{\operatorname{card}}

\DeclareMathOperator{\interior}{int}
\DeclareMathOperator{\Inv}{Inv}
\DeclareMathOperator{\FRRV}{FRRV}
\DeclareMathOperator{\NFRRV}{NFRRV}

\newcommand{\bR}{{\mathbb R}}

\newcommand{\bZ}{{\mathbb Z}}
\newcommand{\cA}{{\mathcal A}}

\newcommand{\cF}{{\mathcal F}}
\newcommand{\cG}{{\mathcal G}}
\newcommand{\cN}{{\mathcal N}}

\newcounter{question}

\begin{document}



\title[Topological-numerical analysis of a 2D neuron model]{Topological-numerical analysis of a two-dimensional discrete neuron model}

\author[P.\ Pilarczyk]{Pawe\l{} Pilarczyk}
\address{
Faculty of Applied Physics and  Mathematics \& Digital Technologies Center,
Gda\'{n}sk University of Technology,
ul.\ Narutowicza 11/12,
80-233 Gda\'{n}sk, Poland.
\newline
Orcid: \href{https://orcid.org/0000-0003-0597-697X}{0000-0003-0597-697X}
}

\author[J.\ Signerska-Rynkowska]{Justyna Signerska-Rynkowska$^*$}
\thanks{$^*$Corresponding author.}
\address{
Dioscuri Centre in Topological Data Analysis, 
Institute of Mathematics of the Polish Academy of Sciences,
ul.\ \'{S}niadeckich 8,
00-656 Warszawa, Poland;
\newline 
and Faculty of Applied Physics and Mathematics \& BioTechMed Center,
Gda\'{n}sk University of Technology,
ul.\ Narutowicza 11/12,
80-233 Gda\'{n}sk, Poland.
\newline
Orcid: \href{https://orcid.org/0000-0002-9704-0425}{0000-0002-9704-0425}
}

\author[G.\ Graff]{Grzegorz Graff}
\address{
Faculty of Applied Physics and Mathematics \& BioTechMed Center,
Gda\'{n}sk University of Technology,
ul.\ Narutowicza 11/12,
80-233 Gda\'{n}sk, Poland.
\newline
Orcid: \href{https://orcid.org/0000-0001-5670-5729}{0000-0001-5670-5729}
}


\begin{abstract}
We conduct computer-assisted analysis of the two-dimensional model of a neuron introduced by Chialvo in 1995 (\textit{Chaos, Solitons \& Fractals} 5, 461--479).
We apply the method for rigorous analysis of global dynamics based on a set-oriented topological approach, introduced by Arai et al.\ in 2009 (\textit{SIAM J.\ Appl.\ Dyn.\ Syst.}\ 8, 757--789) and improved and expanded afterwards.
Additionally, we introduce a new algorithm to analyze the return times inside a chain recurrent set. Based on this analysis, together with the information on the size of the chain recurrent set, we develop a new method that allows one to determine subsets of parameters for which chaotic dynamics may appear.
This approach can be applied to a variety of dynamical systems, and we discuss some of its practical aspects.
The data and the software described in the paper are available at \url{http://www.pawelpilarczyk.com/neuron/}.
\end{abstract}

\maketitle


\noindent
\textbf{In the last three decades, various discrete models of a single neuron were introduced, aimed at reflecting the dynamics of neural processes. Unfortunately, analytical methods offer limited insight into the nature of some phenomena encountered by such models. In this paper, we study the classical multi-parameter Chialvo model by means of a novel topological method that uses set-oriented rigorous numerics combined with computational topology. We enrich the existing tools with a new approach that we call Finite Resolution Recurrence. We obtain a comprehensive picture of global dynamics of the model, and we reveal its bifurcation structure. We combine the recurrence analysis with machine learning methods in order to detect parameter ranges that yield chaotic behavior.}


\section{Introduction}
\label{sec:intro}

 
With the increasing capabilities of contemporary computers, it is possible to apply more and more computationally demanding methods to the analysis of dynamical systems. Such methods may provide comprehensive overview of the dynamics on the one hand, and thorough insight into specific features of the system on the other hand.

In this paper, we discuss an application of a computationally advanced method for the analysis of global dynamics of a system with many parameters. The method was originally introduced in \cite{Arai}, and now we enhance it by introducing Finite Resolution Recurrence (FRR) analysis, as explained in Section~\ref{sec:frr}. We apply this method to the two-dimensional discrete-time semi-dynamical system introduced by Chialvo in \cite{Chialvo} for modeling a single neuron. We describe this model in Section~\ref{sec:model}.


\FloatBarrier
\subsection{Goals and main results}
\label{sec:goals}

The goals of the paper are twofold. First, we aim at obtaining specific results on the Chialvo dynamical model of a neuron that might be of interest to computational neuroscientists. Second, with this motivation in mind, we develop new numerical-topological methods that can be applied to a wide variety of dynamical systems; these methods are thus of importance to the community of applied scientists interested in computational analysis of mathematical models. The remainder of the Introduction section contains an overview of both achievements.

Our first major result regarding the Chialvo model is that we give complete description of bifurcation patterns within a wide range of parameters (see Figure \ref{fig:n12cont}), together with the information on the dynamics inside each continuation class, expressed by means of the Conley Index and Morse decomposition (as explained in Section~\ref{sec:indiv}). We also determine the changes in dynamics caused by changes in parameter values. These results are broadly discussed in Section~\ref{sec:bifurcations} and summarized in Figure~\ref{fig:n12graphs}, and may be perceived as our main finding about the Chialvo model. Let us remark that this part uses interval arithmetic in the computations and the obtained results are rigorous (computer-assisted proof).

The second main result of the paper regarding the Chialvo model is the indication of possible ranges of parameters in which one may expect chaotic dynamics (and other ranges in which one should not expect it). This is achieved by introducing a new method that we call \emph{Finite Resolution Recurrence analysis}; see Section~\ref{sec:recurrence}. We use it for classifying the type of dynamics with the help of machine learning (DBSCAN clustering) in Section~\ref{sec:recTool}. The result of this part of our research is summarized in Figures \ref{fig:learn25d_h10} and~\ref{fig:learn25d_h00} in which we identify six main types of dynamics (including chaos) and the corresponding parameter ranges.
\label{rev:nonrigorous1}
Although computation of Finite Resolution Recurrence is rigorous and one can use it to prove certain features of the dynamical system (as we explain in Section~\ref{sec:frr}), in this part of the paper we use it in a heuristic way to draw non-rigorous yet meaningful conclusions.
\label{rev:spikingMentioned}
Our other heuristic result is the identification of regions of parameters in which spiking-bursting oscillations are likely to appear; this result is discussed in Section~\ref{sec:sizes}.

Our numerical-topological methods are briefly introduced in Section~\ref{sec:approachOverview}, and some of their advantages over the ``classical'' approach are gathered in Section~\ref{sec:advantages}. We emphasize the fact that our methods are universal, i.e., the scope of their applicability is not limited to the Chialvo map, but they can also be applied to various other kinds of dynamical systems.


\FloatBarrier
\subsection{Overview of our numerical-topological approach}
\label{sec:approachOverview}

Our approach uses rigorous numerical methods and a topological approach based on the Conley index and Morse decompositions, and provides mathematically validated results concerning the qualitative dynamics of the system. The main idea is to cover the phase space (a subset of $\bR^n$) by means of a rectangular grid ($n$-dimensional rectangular boxes), and to use interval arithmetic to compute an outer estimate of the map on the grid elements. This construction gives rise to a directed graph, and fast graph algorithms allow one to enclose all the recurrent dynamics in bounded subsets, further called Morse sets, built of the grid elements, so that the dynamics outside the collection of these subsets is gradient-like. The entire range of parameters under consideration is split into classes in such a way that parameters within one class yield equivalent dynamics. We outline this method in Sections \ref{sec:glob}--\ref{sec:indiv}. We show its practical application to obtain a comprehensive overview of the different types of dynamics that appear in the Chialvo model in Sections~\ref{sec:appl}--\ref{sec:bifurcations}.

Since existence of chaotic dynamics implies recurrence in large areas of the phase space (existence of large ``strange attractors''), construction of an outer estimate for the chain recurrent set results in this case in just one large isolating neighborhood, and therefore  the approach based on constructing a Morse decomposition provides very little information on the actual dynamics. In order to address this problem, we introduce new algorithmic methods for the analysis of the directed graph that represents the map in order to get insight into the dynamics \emph{inside} this kind of a large Morse set. We consider this a non-trivial extension of the method described in \cite{Arai} that provides new and important information on the dynamics. In particular, we introduce the notion of Finite Resolution Recurrence (FRR for short) in Section~\ref{sec:frr}, and we show its application to a few cases in the Chialvo model in Section~\ref{sec:recComp}. We then propose (in Sections~\ref{sec:FinVar}--\ref{sec:NormVar}) to analyze the variation of FRR values inside the large Morse set,  and we conduct comprehensive analysis of Normalized FRR Variation (NFRRV for short) in Section~\ref{sec:varComp}.

Finally, in Section \ref{sec:recTool}, we develop certain heuristic indicators of chaotic dynamics that are based on the FRR analysis and apply them to the Chialvo model. The results in this section are no longer rigorous; these are heuristics supported by machine learning and numerical evidence.
\label{rev:nonrigorous2}

\begin{figure}[htbp]
\centering
\includegraphics[width=0.75\textwidth]{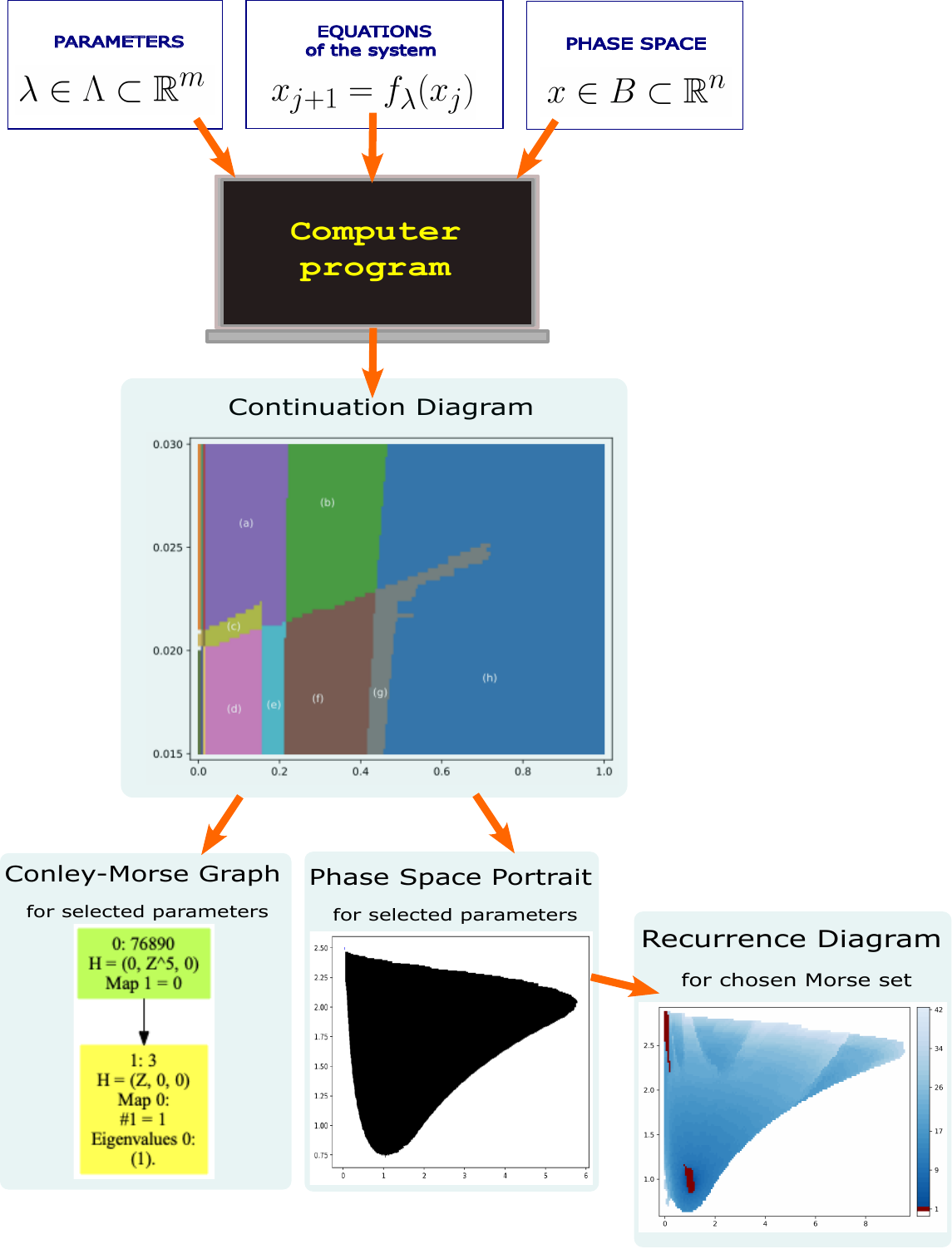} 
\caption{\label{fig:Summary}%
Schematic illustration of the main steps of our method. The equations of the system, together with a chosen range of parameters $\Lambda\subset \mathbb{R}^m$ and a specification of the phase space ($B\subset \mathbb{R}^n$) constitute input to a computer program: an implementation of the method described in Sections \ref{sec:glob}, \ref{sec:indiv}, \ref{sec:frr} and Appendix~\ref{app:recurrence}. The computer program can be treated as a ``black box'' that produces desired information about the system on its output. The first thing to obtain is a Continuation Diagram that partitions the parameter space $\Lambda$ into classes corresponding to different types of dynamics. Next, for each particular value of the parameters, a Conley-Morse Graph and the corresponding Phase Space Portrait are obtained. This provides information on invariant sets. Finally, the dynamics within individual Morse sets can be examined in more detail by means of Recurrence Diagrams that allow one to deduce possible existence of periodic attractors and chaotic dynamics.}
\end{figure}

One could summarize the main ideas of our method for comprehensive analysis of a dynamical system at finite resolution in the following way (see Figure~\ref{fig:Summary}):
\begin{enumerate}
\setlength{\itemsep}{0pt}
\renewcommand{\theenumi}{\alph{enumi}}
\item\label{itemSplit}
We split the dynamics into chain recurrent sets (Morse sets)
and the non-recurrent set (gradient-like dynamics).
\item\label{itemConley}
We obtain information about the Morse sets by looking at their boundary (computing their Conley indices); in particular, a nontrivial Conley index implies the existence of a non-empty isolated invariant subset of the Morse set.
\item\label{itemRigorous}
Computation of (\ref{itemSplit}) and (\ref{itemConley}) involves rigorous numerical methods (based on interval arithmetic) and provides mathematically reliable results (computer-assisted proof).
\item\label{itemFRR}
We study the dynamics \emph{inside} the Morse sets using a new method: Finite Resolution Recurrence analysis (also providing rigorous results).
\item\label{itemAll}
We conduct (\ref{itemSplit})--(\ref{itemFRR}) for a grid of parameters of the system varying in some bounded ranges, and we gather the information in order to classify the types of dynamics encountered (rigorous results), and search for spiking-bursting oscillations and chaotic dynamics (heuristic non-rigorous results).
\end{enumerate}
In particular, by applying our approach to the classical Chialvo model \eqref{main} we are able to obtain precise and comprehensive description of the dynamics for a large range of parameters considered in \cite{Chialvo}.


\FloatBarrier
\subsection{Advantages of our method in comparison to ``classical'' approach}
\label{sec:advantages}

As it will be made clear in the sequel, the set-oriented topological method that we apply has several advantages over purely analytical methods, and over plain numerical simulations as well. One could argue that analytical methods typically focus on finding equilibria of the system and determining their stability. On the other hand, numerical simulations are usually limited to iterating individual trajectories and thus their ability is limited to finding stable invariant sets only, not to mention their vulnerability to round-off or approximation errors. In contrast to this, our approach detects all kinds of recurrent dynamics (also unstable) in a given region of the phase space, and provides mathematically reliable results.

It is also worth pointing out that introducing advanced methods for investigating invariant sets and their structure is especially desirable in discrete-time systems already in dimension $2$, such as the Chialvo model. The reason is that such systems are much more demanding than their ODE ``counterparts.'' For example, in dimension~$2$, the information on the invariant (limit) sets in continuous-time systems can be concluded from the Poincar\'{e}-Bendixson Theorem, the shape of stable and unstable manifolds of saddle fixed points, and other elementary considerations. Indeed, as Chialvo noticed in~\cite{Chialvo}, ``Trajectories associated with iterated maps are sets of discrete points, and not continuous curves as in ODE. In two-dimensional ODE, orbits or stable and unstable manifolds partition the phase space in distinct compact subsets with their specific attractors. Structure of the stable sets might be more complicated for 2D iterated maps.'' 

On the other hand, in the analysis of one-dimensional discrete models, one can benefit from the theory of circle maps or the theory of interval maps; both have undergone rapid development in recent decades. In particular, the theory of $S$-unimodal maps can be successfully applied to obtain rigorous results for the one-dimensional Chialvo model \cite{Chialvo1dim}. Unfortunately, in the discrete setting, increasing the dimension from one to two makes the analysis considerably harder, as no such powerful analytical tools exist for discrete models in two dimensions. Hence there is a strong demand for reliable computational techniques in studying (discrete) higher dimensional systems. This demand has been one of our main motivations for conducting the research described in this paper.


\FloatBarrier
\subsection{Structure of the paper}
\label{sec:structure}

The core of the paper is split into three sections. Introduction of theoretical basis and description of computational methods is directly followed by application to the Chialvo model.

In Section~\ref{sec:model}, we describe a $2$-dimensional discrete-time dynamical system introduced by Chialvo for modeling an individual neuron.

In Section~\ref{sec:globaldyn}, we explain the set-oriented topological method for comprehensive analysis of a dynamical system, and we apply it to the Chialvo model. In particular, we explain the various kinds of global dynamics that we encountered in the phase space across the analyzed ranges of parameters, we describe possible bifurcations found in the system, and we give heuristics on where one could search for spiking-bursting oscillations and chaotic dynamics.

In Section~\ref{sec:recurrence}, we introduce the Finite Resolution Recurrence (FRR) and its variation (FRRV), also normalized (NFRRV) as new mathematical tools for deep analysis of recurrent dynamics at limited resolution. We show the results of applying this method to the Chialvo model.

Finally, in Section~\ref{sec:recTool}, we use FRR as a tool in classification of dynamics, and we demonstrate its usefulness as an indicator of the existence of chaotic dynamics.


\section{Model}
\label{sec:model}
The following model of a single neuron was proposed by D.\ Chialvo in \cite{Chialvo}:
\begin{equation}
\label{main}
\left\{\begin{array}{rcl}
x_{n+1} & = & x_n^2 \exp(y_n-x_n) + k,
\\
y_{n+1} & = & ay_n -bx_n + c.
\end{array}
\right.
\end{equation}
In this model, $x$ stands for the membrane (voltage) potential of a neuron. It is the most important dynamical variable in all neuron models. However, in order to model neuron kinetics in a more realistic way than by means of a single variable, at least one other dynamical variable must be included in the model. Therefore, the system \eqref{main} contains also $y$ that acts as a recovery-like variable. There are four real parameters in this model: $k$ which can be interpreted as an additive perturbation or a current input the neuron is receiving, $a>0$ which is the time constant of neuron's recovery, the activation (voltage) dependence of the recovery process $b>0$, and the offset $c$.

This discrete model, in which $x_n$ and $y_n$ are values of the voltage and the recovery variable at the consecutive time units $n$, belongs to the class of so-called map-based models. Such models have received a lot of attention recently, and include the famous Rulkov models \cite{rulkov2001,rulkov2002,ShilnikovRulkov2004} \label{extra} and many others (see also review articles \cite{courbage2010} and  \cite{ibarz2011}). For completeness, we also mention the fact that neurons can be modeled by continuous dynamical systems, i.e., ordinary differential equations, dating back to the pioneering work of Hodgkin and Huxley \cite{Hodgkin}, or by hybrid systems (see e.g. \cite{BretteGerstner,izi2003,wild1,typeIII,mma,touboul2009}).  Models taking into account the propagation of the voltage through synapses or models of neural networks often incorporate PDEs and stochastic processes (see e.g. \cite{Burkitt2006,torres2021}). 
Although some might consider map-based models too simplified from the biological point of view, their biological relevance in fact can be sometimes satisfactory. Their important advantage is that they are often computationally plausible as components of larger systems. Due to the tremendous complexity of real neuronal systems, map-based models appeared as models that are simple enough to be dealt with, yet able to capture the most relevant properties of the cell. Map-based models can sometimes be seen as discretizations of ODE-based models.  

It seems that the Chialvo model \eqref{main} is not a direct discretization of any of the popular ODE models. However,  as noticed by Chialvo himself (see \cite{Chialvo}), the shape of nullclines reminds that of some two-dimensional ODE excitable systems. Moreover, the equation  \eqref{main} fulfills the most common general form of map-based models of a neuron (compare with \cite{ibarz2011}):
\begin{equation} \label{GeneralMapBased}
\left\{\begin{array}{rcl}
x(t+1) & = & F\big(x(t), k\pm y(t)\big),\\
y(t + 1) & = & y(t) \mp \varepsilon \big(x(t) - q y(t) - \sigma\big).
\end{array}
\right.
\end{equation}

Before we proceed with our analyses, let us briefly summarize main properties of the dynamics of \eqref{main} that have already been described in the literature.
Note that, in general, the overall analysis of the phase plane dynamics for the model \eqref{main} has not been conducted. There are, however, valuable observations for some ranges of parameters supported by  numerical simulations.  

In the paper \cite{Chialvo}, in which this model was introduced, Chialvo discusses only the case $k=0$, and then $k$ of small positive value with two prescribed choices of the other parameters.
For $k=0$, the point $(x_{\textrm{f}\,0},y_{\textrm{f}\,0}):=(0, c/(1-a))$ is always a stable (attracting) fixed point of the system (since the corresponding eigenvalues are $0$ and $a<1$).
For $k\neq 0$, Chialvo \cite{Chialvo} treats only the case when the phase portrait has exactly one equilibrium point (which happens, e.g., for $a=0.89$, $b=0.6$ and $c=0.28$) and treats $k$ as the bifurcation parameter, while $a$, $b$ and $c$ are usually kept constant.  For this particular choice of parameter values and small values of $k$, the unique fixed point is  globally attracting and this parameter regime is referred to as \emph{quiescent-excitable} regime. For larger values of $k$ (e.g., for $k=0.1$), the unique fixed point is no longer stable, and oscillatory solutions might appear.  This phenomenologically corresponds to the bifurcation from  quiescent-excitable to oscillatory solution.  When the value of $k$ is increased a bit  more, chaotic-like behavior was observed in \cite{Chialvo} for some values of the parameters. For example, when $b$ is decreased from $0.6$ to $0.18$, and  $a=0.89$, $c=0.28$, $k=0.03$ then instead of periodic-like oscillations the solution displays chaotic bursting oscillations with large spikes often followed by a few oscillations of smaller amplitude, resembling so-called mixed-mode oscillations (see Figure~10 in \cite{Chialvo}). 

The system \eqref{main} can have up to three fixed points and their  existence and stability as well as bifurcations were studied analytically in \cite{Jing}.
Numerical simulations described in another work \cite{NewPaperOnChialvo} suggest the existence of an interesting structure in the $(a,b)$-parameter space (with $c=0.28$ and $k=0.04$ fixed), including comb-shaped periodic regions (corresponding to period-incrementing bifurcations), Arnold tongue structures (due to the period-doubling bifurcations) and shrimp-shaped structures immersed in large chaotic regions.

Let us also mention the fact that the recent work \cite{Chialvo1dim} studies in detail the dynamics of the reduced Chialvo model, i.e., the evaluation of the membrane voltage given by the first equation in \eqref{main}, with $y_n=\mathrm{const}$ treated as a parameter. These purely analytical studies take advantage of the fact that the one-dimensional map $F_y(x)=x^2\exp(y-x)+k$ of the $x$-variable, restricted to the invariant interval of interest, is unimodal with negative Schwarzian derivative, which makes it possible to use the well-developed theory of S-unimodal maps. 

Despite all these important observations mentioned above, it is clear that the description of the dynamics of the model \eqref{main} in the existing literature is very incomplete. In our research described in this paper, we aimed at obtaining a better understanding of the two-dimensional model \eqref{main} in a reasonable parameter range. We focused mainly on investigating the set of parameters that covered most of the analyses conducted in \cite{Chialvo}. Specifically,  we fixed $a = 0.89$ and $c = 0.28$, and we made the other two parameters vary. We first studied the range $(b,k) \in \Lambda_1 := [0,1] \times [0,0.2]$ (see Appendix~\ref{app:continuation}), and based on these results we decided to restrict our attention to its sub-region $\Lambda_2 := [0, 1] \times [0.015, 0.030]$ (see Sections \ref{sec:appl}--\ref{sec:bifurcations} for the detailed results). In addition to the typical behavior observed in~\cite{Chialvo}, including attracting points, attractor-repeller pairs consisting of a point and a periodic orbit, and chaotic behavior as well, we detected many regions with other types of interesting dynamics, especially for very small values of the parameters $b$ and $k$.


\section{Set-oriented numerical-topological analysis of global dynamics}
\label{sec:globaldyn}

In this section, we describe the method for computer-assisted analysis of dynamics in a system with a few  parameters, first introduced in \cite{Arai} for discrete-time dynamical systems, and further extended in \cite{Knipl} to flows. We also introduce a new method based on the notion of Finite Resolution Recurrence that provides insight into the dynamics inside chain recurrent sets. This approach provides a considerable improvement, because -- to the best of our knowledge -- in methods based on \cite{Arai} introduced so far this kind of analysis that would reveal the internal structure of chain recurrent sets was never proposed.

The computations are conducted for entire intervals of parameters, and the results are valid for each individual parameter in the interval. This allows one to determine the dynamical features for entire parameter ranges if those are subdivided into smaller subsets. By using interval arithmetic and controlling the rounding of floating point numbers, the method provides mathematically rigorous results (a.k.a.\ \emph{computer-assisted proof}).

We first describe the set-oriented rigorous numerical method for the computation of Morse decomposition of the dynamics on a given phase space across a fixed range of parameters in Section~\ref{sec:glob}. The first paragraph of that section is a concise description of the process, and the remainder contains all the technical details that can be skipped on the first reading. The result of applying this method to the Chialvo model is described in Section~\ref{sec:appl}, together with information on how to use the database available for interactive on-line viewing at~\cite{www}, and the technical details are gathered in Appendix~\ref{app:continuation}. Then in Section~\ref{sec:indiv}, we explain the topological approach to the analysis of individual components of recurrent dynamics found in the previous step (Morse sets) by means of the Conley index. The description is aimed at non-users of the Conley index theory and provides information necessary to understand our results discussed in Section~\ref{sec:bifurcations}, in which we provide a comprehensive overview of all the types of dynamics that we found in the Chialvo model. Finally, in Section~\ref{sec:sizes}, we provide a heuristic method for using the results of computations conducted in Section~\ref{sec:appl} to detect regions of parameters for which spiking-bursting oscillations or chaotic dynamics might appear.


\FloatBarrier
\subsection{Automatic analysis of global dynamics}
\label{sec:glob}

While numerical simulations based on computing individual trajectories may provide some insight into the dynamics, considerably better understanding may be achieved by set-oriented methods in which entire sets are iterated by the dynamical system. One of the first software packages that used this approach was GAIO~\cite{GAIO}. The first step is to partition the phase space into a collection of bounded sets with simple structure (such as squares or cubes), further called \emph{grid elements}. By considering images of these sets, one can represent a map that generates a discrete-time dynamical system as a directed graph on grid elements. Numerical methods based on interval arithmetic provide means for computing an outer enclosure of the map rigorously and effectively. And here comes the key idea. Effective graph algorithms applied to such a representation of \label{themap} the map make it possible to capture all
the chain recurrent dynamics contained in a collection of subsets of the phase
space, called  \emph{Morse sets}. In particular, this construction proves that the dynamics in the remaining part of the phase space is gradient-like (see also \cite{BK2006,KMV2005}). By determining possible connections between the chain recurrent components, one constructs so-called \emph{Morse decomposition}, and the Conley index \cite{Conley} provides additional information about the invariant part of the Morse sets. Finally, the set of all the possible values of parameters within prescribed ranges is split into subsets of parameters that yield equivalent Morse decompositions related by continuation, thus called \emph{continuation classes}.

The remainder of this subsection contains formal definitions of what has just been explained intuitively in the paragraph above, and can be skipped on the first reading.

Formally, let $X$ be a topological space,
and let $f \colon X \to X$ be a continuous map.
$S \subset X$ is called an {\em invariant set} with respect to $f$
if $f (S) = S$.
The {\em invariant part} of a set $N \subset X$
is an invariant set defined as $\Inv N := \bigcup \{S \subset N : f (S) = S\}$.
An {\em isolating neighborhood}
is a compact set $N$ whose invariant part is contained in its interior:
$\Inv N \subset \interior N$.
A set $S$ is called an {\em isolated invariant set}
if $S = \Inv N$ for some isolating neighborhood $N$.

A {\em Morse decomposition} of $X$
with respect to $f$ is a finite collection
of disjoint isolated invariant sets (called {\em Morse sets})
$S_1,\ldots,S_p$ with a strict partial ordering $\prec$
on the index set $\{1,\ldots,p\}$
such that for every $x \in X \setminus (S_1 \cup \cdots \cup S_p)$
and for every orbit $\{\gamma_k\}_{k \in \bZ}$
(that is, a bi-infinite sequence for which $f (\gamma_k) = \gamma_{k + 1}$) such that $\gamma_0 = x$
there exist indices $i \prec j$
such that $\gamma_k \to S_i$ as $k \to \infty$
and $\gamma_k \to S_j$ as $k \to -\infty$.

Since it is not possible, in general, to construct numerically a valid Morse decomposition of a compact set $B \subset \bR^n$,
we construct isolating neighborhoods of the Morse sets instead.
This is a family of isolating neighborhoods $N_1, \ldots, N_p \subset  B$
with a strict partial ordering $\prec$ on the set of their indices
such that the family $\{S_i := \Inv N_i : i = 1, \ldots, p\}$
forms a Morse decomposition of $\Inv B$
with the ordering $\prec$.
The sets $N_i$, $i = 1, \ldots, p$, will be called {\em numerical Morse sets}, and the collection $N_1, \ldots, N_p$ is then a {\em numerical Morse decomposition}.

We visualize a numerical Morse decomposition by means of a directed graph that corresponds to the transitive reduction of the relation $\prec$. Vertices in this graph correspond to the numerical Morse sets, and a path from $N_i$ to $N_j$
indicates the possibility of existence of a connecting orbit between them.

We construct numerical Morse sets as finite unions of small rectangles in $\bR^n$ whose vertices form a regular mesh. Specifically, a {\em rectangular set} in $\bR^n$ is a product of compact intervals.
Given a rectangular set
\[
B = [a_1, a_1 + \delta_1] \times
\cdots \times [a_n, a_n + \delta_n] \subset \bR^n
\]
and integer numbers $d_1, \ldots, d_n > 0$,
the set
\[
\cG_{d_1, \ldots, d_n} (B) :=
\bigg\{
\prod_{i=1}^{n} [a_i + \frac{j_i}{d_{i}} \delta_i,
a_i + \frac{j_i + 1}{d_{i}} \delta_i] : \\
j_i \in \{0, \ldots, d_i - 1\},
i \in \{1, \ldots, n\}
\bigg\}
\]
is called the
{\em $d_1 \times \cdots \times d_n$ uniform rectangular grid in B}.
The grid elements are referred to by the $n$-tuples $(j_1, \ldots, j_n)$.
The $n$-tuple of integers $(d_1, \ldots, d_n)$
is called the {\em resolution} in $B$. We shall often write $\cG(B)$ instead of $\cG_{d_1, \ldots, d_n} (B)$ for short.

Note that in the planar case, a numerical Morse decomposition in $B \subset \bR^2$ can be visualized as a digital raster image whose pixels correspond to the individual boxes in each of the numerical Morse sets. For convenience, each numerical Morse set can be plotted with a different color. Obviously, this kind of visualization should be accompanied by a graph that shows the ordering $\prec$.

A multivalued map $\cF \colon \cG(B) \multimap \cG(R)$, where $B \subset R \subset \bR^n$ and $\cG(R)$ is a uniform rectangular grid containing $\cG(B)$,
is called a \emph{representation} of a continuous map $f \colon B \to R$ if the image $f (Q)$ of every grid element $Q \in \cG(B)$ is contained in the interior of the union of grid elements in $\cF(Q)$. If $|\cA|$ denotes the union of all the grid elements that belong to the set $\cA \subset \cG(R)$ then this condition can be written as follows:
\begin{equation}
\label{representation}
f(Q) \subset \interior |\cF(Q)| \text{ for every } Q \in \cG(B).
\end{equation}

A representation corresponds to a directed graph $G = (V,E)$ whose vertices are grid elements and directed edges are defined by the mapping $\cF$ as follows: $(P,Q) \in E \iff Q \in \cF(P)$.
It is a remarkable fact that the decomposition of $G$ into \emph{strongly connected path components} (maximal collections of vertices connected in both directions by paths of nonzero length) yields a numerical Morse decomposition in $B$, provided that $N_i \subset \interior B$ for all the numerical Morse sets $N_i$; see \cite{Arai,BK2006,KMV2005} for justification.

Now consider a dynamical system that depends on $m$ parameters.
Consider a rectangular set $\Lambda \in \bR^m$ of all the $m$ parameter values of interest. Take a uniform rectangular grid $\cG_{s_1,\ldots,s_m}(\Lambda)$ for some positive integers $s_1,\ldots,s_m$.
Using interval arithmetic, one can compute a representation $\cF_L$ valid for the maps $f_\lambda$ for all the parameters $\lambda \in L$.
Then the numerical Morse decomposition computed for $F_L$ yields a collection of isolating neighborhoods of a Morse decomposition for each $f_\lambda$, where $\lambda \in L$.

Given two parameter boxes $L_1,L_2 \in \cG(\Lambda)$, we use the clutching graph introduced in \cite[\S 3.2]{Arai} to check if the numerical Morse sets in the computed two numerical Morse decompositions are in one-to-one correspondence. If this is the case then continuation of Morse decompositions has been proved and we consider the dynamics found for the parameter boxes $L_1,L_2$ equivalent. A visualization of the collection of equivalence classes with respect to this relation is called a \emph{continuation diagram}.


\FloatBarrier
\subsection{Analysis of dynamics in individual Morse sets using the Conley index}
\label{sec:indiv}

Informally speaking, the dynamics in the Morse sets could be of three types. In the first type, all trajectories that enter the set stay there forever in forward time, like in the examples shown in Figure~\ref{fig:indStable}. In the second type, there are some trajectories that stay in the Morse set in forward time, but there are also some other trajectories that exit the set. This situation is shown in the examples in Figure~\ref{fig:indUnstable}. Finally, it may be the case that all trajectories that enter the set will leave it in forward time, and thus there is no trajectory that stays inside forever. Two such examples are shown in Figure~\ref{fig:indTrivial}.

\begin{figure}[b]
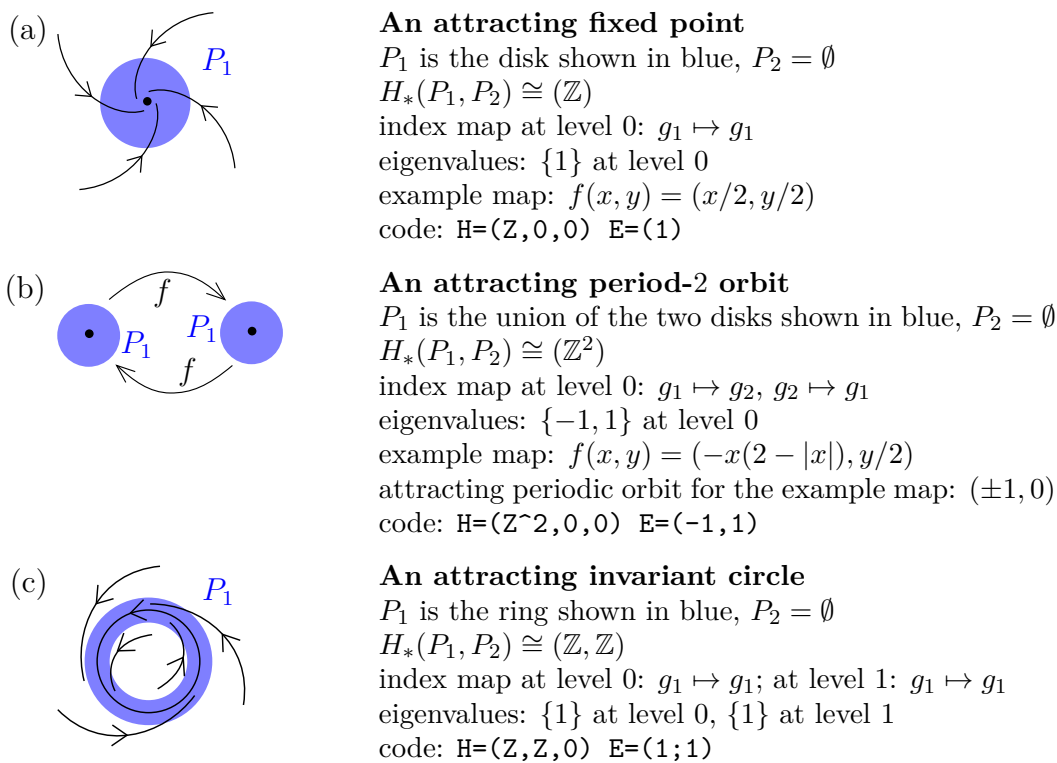

\begin{minipage}[t]{3cm}
\vspace{0pt}
\setlength{\unitlength}{2.5cm}%
\begin{picture}(1,0.94787965)%
\put(0,0){\includegraphics[width=\unitlength,page=1]{\detokenize{ind_attr_point}}}%
\put(0.76666196,0.64080595){\color[rgb]{0,0,1}$P_1$}%
\put(-0.05,0.93){\makebox(0,0)[rt]{(a)}}%
\end{picture}%
\end{minipage} \hspace{1cm}
\begin{minipage}[t]{9cm}
\vspace{0pt}\small
\textbf{An attracting fixed point} \\
$P_1$ is the disk shown in blue,
$P_2 = \emptyset$ \\
$H_*(P_1, P_2) \cong (\bZ)$ \\
index map at level $0$: $g_1 \mapsto g_1$ \\
eigenvalues: $\{1\}$ at level $0$ \\
example map: $f(x,y)=(x/2,y/2)$ \\
code: \texttt{H=(Z,0,0) E=(1)}
\end{minipage}
\vskip 10pt
\begin{minipage}[t]{3cm}
\vspace{0pt}
\setlength{\unitlength}{3cm}%
\begin{picture}(1,0.54164849)%
\put(0,0){\includegraphics[width=\unitlength,page=1]{\detokenize{ind_attr_period2}}}%
\put(0.42437613,0.41){\color[rgb]{0,0,0}$f$}%
\put(0,0){\includegraphics[width=\unitlength,page=2]{\detokenize{ind_attr_period2}}}%
\put(0.53536928,0.06){\color[rgb]{0,0,0}$f$}%
\put(0.28,0.18){\color[rgb]{0,0,1}$P_1$}%
\put(0.57,0.25){\color[rgb]{0,0,1}$P_1$}%
\put(-0.05,0.53){\makebox(0,0)[rt]{(b)}}%
\end{picture}%
\end{minipage} \hspace{1cm}
\begin{minipage}[t]{9cm}
\vspace{0pt}\small
\textbf{An attracting period-$2$ orbit} \\
$P_1$ is the union of the two disks shown in blue,
$P_2 = \emptyset$ \\
$H_*(P_1, P_2) \cong (\bZ^2)$ \\
index map at level $0$: $g_1 \mapsto g_2$, $g_2 \mapsto g_1$ \\
eigenvalues: $\{-1, 1\}$ at level $0$ \\
example map: $f(x,y)=(-x(2-|x|),y/2)$ \\
attracting periodic orbit for the example map: $(\pm 1,0)$ \\
code: \texttt{H=(Z\symbol{94}2,0,0) E=(-1,1)}
\end{minipage}
\vskip 10pt
\begin{minipage}[t]{3cm}
\vspace{0pt}
\setlength{\unitlength}{2.5cm}%
\begin{picture}(1,0.96498198)%
\put(0,0){\includegraphics[width=\unitlength,page=1]{\detokenize{ind_attr_circle}}}%
\put(0.77,0.78){\color[rgb]{0,0,1}$P_1$}%
\put(0,0){\includegraphics[width=\unitlength,page=2]{\detokenize{ind_attr_circle}}}%
\put(-0.05,0.95){\makebox(0,0)[rt]{(c)}}%
\end{picture}%
\end{minipage} \hspace{1cm}
\begin{minipage}[t]{9cm}
\vspace{0pt}\small
\textbf{An attracting invariant circle} \\
$P_1$ is the ring shown in blue,
$P_2 = \emptyset$ \\
$H_*(P_1, P_2) \cong (\bZ, \bZ)$ \\
index map at level $0$: $g_1 \mapsto g_1$;
at level $1$: $g_1 \mapsto g_1$ \\
eigenvalues: $\{1\}$ at level $0$, $\{1\}$ at level $1$ \\
code: \texttt{H=(Z,Z,0) E=(1;1)}
\end{minipage}
\vskip 10pt
\caption{\label{fig:indStable}
Typical Conley indices for stable isolated invariant sets that appear in actual applied dynamical systems. The examples (a) and (c) may come from a time-$t$ map for a flow (the index map is thus the identity).}
\end{figure}

\begin{figure}[htbp]
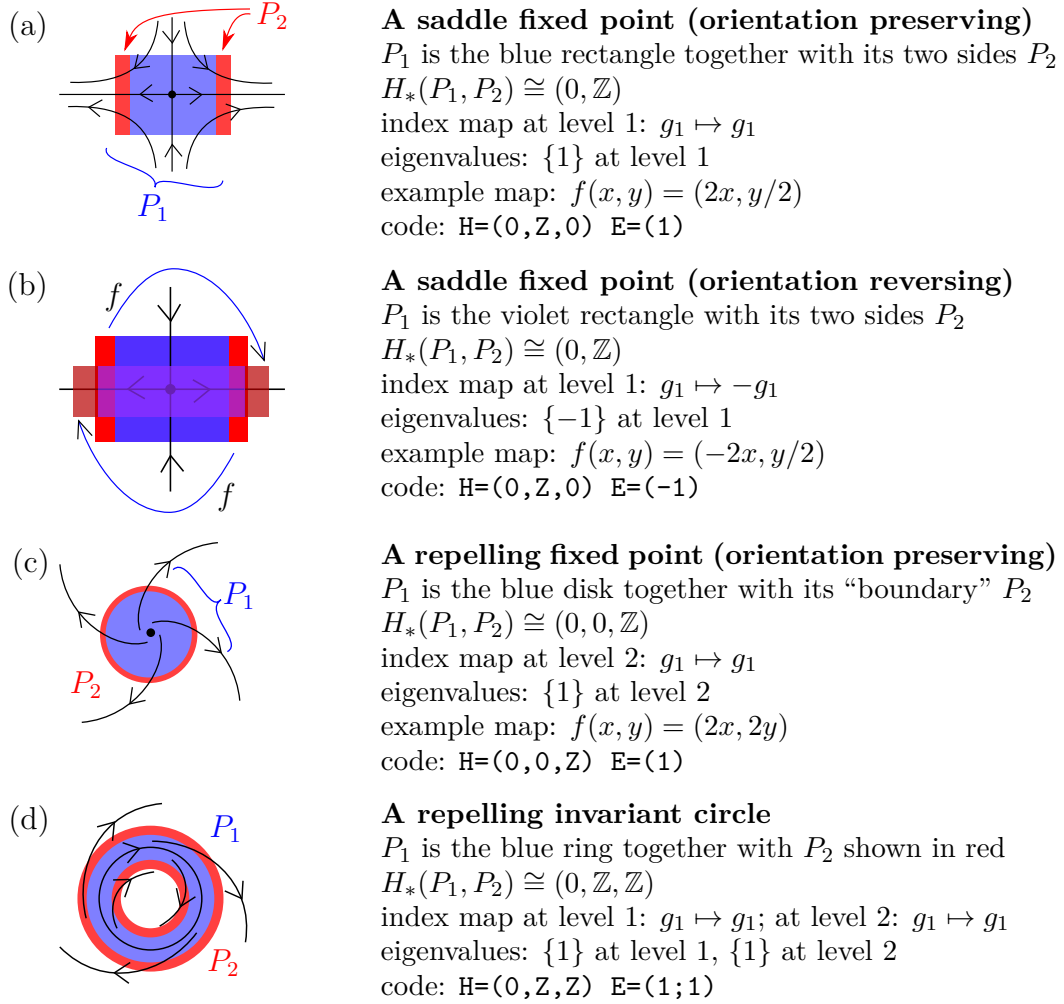

\begin{minipage}[t]{3cm}
\vspace{0pt}
\setlength{\unitlength}{3cm}%
\begin{picture}(1,0.9258462)%
\put(0,0){\includegraphics[width=\unitlength,page=1]{\detokenize{ind_saddle_point}}}%
\put(0.34310559,0){\color[rgb]{0,0,1}$P_1$}%
\put(0,0){\includegraphics[width=\unitlength,page=2]{\detokenize{ind_saddle_point}}}%
\put(0.87230344,0.85){\color[rgb]{1,0,0}$P_2$}%
\put(0,0){\includegraphics[width=\unitlength,page=3]{\detokenize{ind_saddle_point}}}%
\put(-0.05,0.91){\makebox(0,0)[rt]{(a)}}%
\end{picture}%
\end{minipage} \hspace{1cm}
\begin{minipage}[t]{9cm}
\vspace{0pt}\small
\textbf{A saddle fixed point (orientation preserving)} \\
$P_1$ is the blue rectangle
together with its two sides $P_2$ \\
$H_*(P_1, P_2) \cong (0, \bZ)$ \\
index map at level $1$: $g_1 \mapsto g_1$ \\
eigenvalues: $\{1\}$ at level $1$ \\
example map: $f(x,y)=(2x,y/2)$ \\
code: \texttt{H=(0,Z,0) E=(1)}
\end{minipage}
\vskip 10pt
\begin{minipage}[t]{3cm}
\vspace{0pt}
\setlength{\unitlength}{3cm}%
\begin{picture}(1,1.08564221)%
\put(0,0){\includegraphics[width=\unitlength,page=1]{\detokenize{ind_saddle_point_flip}}}%
\put(0.20,0.92362924){\color[rgb]{0,0,0}$f$}%
\put(0,0){\includegraphics[width=\unitlength,page=2]{\detokenize{ind_saddle_point_flip}}}%
\put(0.70,0.03509055){\color[rgb]{0,0,0}$f$}%
\put(-0.05,1.07){\makebox(0,0)[rt]{(b)}}%
\end{picture}%
\end{minipage} \hspace{1cm}
\begin{minipage}[t]{9cm}
\vspace{0pt}\small
\textbf{A saddle fixed point (orientation reversing)} \\
$P_1$ is the violet rectangle with its two sides $P_2$ \\
$H_*(P_1, P_2) \cong (0, \bZ)$ \\
index map at level $1$: $g_1 \mapsto -g_1$ \\
eigenvalues: $\{-1\}$ at level $1$ \\
example map: $f(x,y)=(-2x,y/2)$ \\
code: \texttt{H=(0,Z,0) E=(-1)}
\end{minipage}
\vskip 10pt
\begin{minipage}[t]{3cm}
\vspace{0pt}
\setlength{\unitlength}{2.5cm}%
\begin{picture}(1,0.96318676)%
\put(0,0){\includegraphics[width=\unitlength,page=1]{\detokenize{ind_repeller_point}}}%
\put(0.86899565,0.63150243){\color[rgb]{0,0,1}$P_1$}%
\put(0.05678483,0.15970996){\color[rgb]{1,0,0}$P_2$}%
\put(0,0){\includegraphics[width=\unitlength,page=2]{\detokenize{ind_repeller_point}}}%
\put(-0.05,0.93){\makebox(0,0)[rt]{(c)}}%
\end{picture}%
\end{minipage} \hspace{1cm}
\begin{minipage}[t]{9cm}
\vspace{0pt}\small
\textbf{A repelling fixed point (orientation preserving)} \\
$P_1$ is the blue disk
together with its ``boundary'' $P_2$ \\
$H_*(P_1, P_2) \cong (0, 0, \bZ)$ \\
index map at level $2$: $g_1 \mapsto g_1$ \\
eigenvalues: $\{1\}$ at level $2$ \\
example map: $f(x,y)=(2x,2y)$ \\
code: \texttt{H=(0,0,Z) E=(1)}
\end{minipage}
\vskip 10pt
\begin{minipage}[t]{3cm}
\vspace{0pt}
\setlength{\unitlength}{2.5cm}%
\begin{picture}(1,0.9635575)%
\put(0,0){\includegraphics[width=\unitlength,page=1]{\detokenize{ind_repeller_circle}}}%
\put(0.79922484,0.75694461){\color[rgb]{0,0,1}\makebox(0,0)[lb]{$P_1$}}%
\put(0,0){\includegraphics[width=\unitlength,page=2]{\detokenize{ind_repeller_circle}}}%
\put(0.78701236,0.06837517){\color[rgb]{1,0,0}$P_2$}%
\put(-0.05,0.95){\makebox(0,0)[rt]{(d)}}%
\end{picture}%
\end{minipage} \hspace{1cm}
\begin{minipage}[t]{9cm}
\vspace{0pt}\small
\textbf{A repelling invariant circle} \\
$P_1$ is the blue ring
together with $P_2$ shown in red \\
$H_*(P_1, P_2) \cong (0, \bZ, \bZ)$ \\
index map at level $1$: $g_1 \mapsto g_1$;
at level $2$: $g_1 \mapsto g_1$ \\
eigenvalues: $\{1\}$ at level $1$, $\{1\}$ at level $2$ \\
code: \texttt{H=(0,Z,Z) E=(1;1)}
\end{minipage}
\vskip 10pt
\caption{\label{fig:indUnstable}
Typical Conley indices for unstable isolated invariant sets that appear in actual applied dynamical systems. The examples (a), (c), (d) may come from a time-$t$ map for a flow (the index map is thus the identity). The map in (b) flips the index pair horizontally and squeezes it vertically.}
\end{figure}

\begin{figure}[htbp]
\begin{minipage}[t]{3cm}
\vspace{0pt}
\setlength{\unitlength}{3cm}%
\begin{picture}(1,1.03732889)%
\put(0,0){\includegraphics[width=\unitlength,page=1]{\detokenize{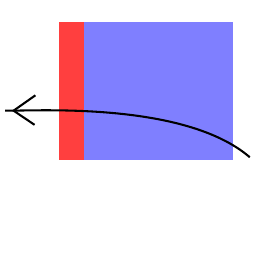}}}%
\put(0.47,0.07){\color[rgb]{0,0,1}\makebox(0,0)[lb]{$P_1$}}%
\put(0,0){\includegraphics[width=\unitlength,page=2]{\detokenize{ind_trivial_side.pdf}}}%
\put(0.05,0.87){\color[rgb]{1,0,0}$P_2$}%
\put(0,0){\includegraphics[width=\unitlength,page=3]{\detokenize{ind_trivial_side.pdf}}}%
\put(-0.07,1){\makebox(0,0)[rt]{(a)}}%
\end{picture}%
\end{minipage} \hspace{1cm}
\begin{minipage}[t]{9cm}
\vspace{0pt}\small
\textbf{A box with trajectories passing through} \\
$P_1$ is the blue rectangle together with its side $P_2$ \\
$H_*(P_1, P_2) \cong 0$ \\
index map: $0$ \\
eigenvalues: $\{0\}$ \\
example map: $f(x,y)=(x-1,y)$ \\
code: \texttt{H=(0) E=(0)}
\end{minipage}
\vskip 10pt
\begin{minipage}[t]{3cm}
\vspace{0pt}
\setlength{\unitlength}{3cm}%
\begin{picture}(1,0.96318676)%
\put(0,0){\includegraphics[width=\unitlength,page=1]{\detokenize{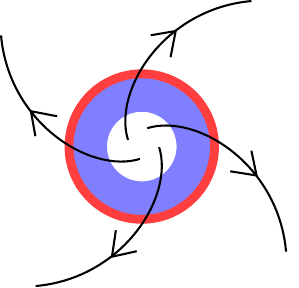}}}%
\put(0.86899565,0.63150243){\color[rgb]{0,0,1}$P_1$}%
\put(0.11,0.22){\color[rgb]{1,0,0}$P_2$}%
\put(0,0){\includegraphics[width=\unitlength,page=2]{\detokenize{ind_trivial_circle.pdf}}}%
\put(-0.07,1){\makebox(0,0)[rt]{(b)}}%
\end{picture}%
\end{minipage} \hspace{1cm}
\begin{minipage}[t]{9cm}
\vspace{0pt}\small
\textbf{A ring with trajectories passing through} \\
$P_1$ is the blue ring
together with its outer belt $P_2$ \\
$H_*(P_1, P_2) \cong 0$ \\
index map: $0$ \\
eigenvalues: $\{0\}$ \\
example map: $f(x,y)=(2x,2y)$ \\
code: \texttt{H=(0) E=(0)}
\end{minipage}
\vskip 10pt
\caption{\label{fig:indTrivial}
Sample index pairs with the trivial Conley index. Note that the invariant part of $P_1 \setminus P_2$ is empty in both cases.}
\end{figure}

Specifically, in order to understand the dynamics in each numerical Morse set $\cN_i$ constructed by the method introduced in Section~\ref{sec:glob}, we check its stability by computing its forward image by $\cF$ and analyzing the part that ``sticks out:'' $\cF(\cN_i) \setminus \cN_i$.
We say that $\cN_i$ is {\em attracting}
if $\cF (\cN_i) \subset \cN_i$; \label{infact} in fact, one can prove that then $|\cN_i|$ contains a non-empty local attractor (cf.\ Lemma~2 in~\cite{Milnor}), 
which justifies this term. 
If $\cF(\cN_i) \not \subset \cN_i$ then
we say that \label{ni} $\cN_i$ is {\em unstable}.
We qualify the kind of instability by computing the Conley index using the approach introduced in \cite{Arai,MMP2005,PS2008}.

The definition of the Conley index is based on the notion of an \emph{index pair}. This is a pair of sets $(P_1,P_2)$ such that $P_1$ covers an isolating neighborhood, and trajectories exit this neighborhood through $P_2$; see e.g.~\cite{Arai} for the precise definition. A few typical Conley indices that \label{appear} appear in our computations are shown in Figures \ref{fig:indStable} and~\ref{fig:indUnstable}.
In the case of a flow, the homological Conley index is merely the relative homology of the index pair.
However, in the case of a map, one also needs to consider the homomorhpism induced in homology by the map on the index pair (denoted here by $H_*(I_P)$), with some reduction applied to it; see e.g. \cite{Szymczak95} for the details. In order to simplify the representation of the Conley index for a map, we compute the non-zero eigenvalues of the index map, which is a weaker but easily computable invariant; see \cite{Arai} for more explanations on this approach.

A selection of typical Conley indices is provided in Figures \ref{fig:indStable} and~\ref{fig:indUnstable}, and two examples of the trivial Conley index are shown in Figure~\ref{fig:indTrivial}, together with the codes that we use in Figures \ref{fig:n12graphs} and~\ref{fig:n12bif}.
In particular, it is important to note that this index has a specific form for a hyperbolic fixed point or a hyperbolic periodic orbit with a $d$-dimensional unstable manifold.
If we encounter one of these specific forms of the index then we say that $\cN_i$ {\em is of type}
of the corresponding periodic point or orbit.
Although in such a case $|\cN_i|$ indeed contains a periodic orbit of the expected period, it may turn out that the stability of that orbit is different, and the dynamics inside the numerical Morse set might be more complicated than it appears from the outside.
In particular, if $|\cN_i| \subset \bR^n$ and $\cN_i$ is of type of a fixed point or a periodic orbit in $\bR^n$
with $n$-dimensional unstable manifold
then we say that $\cN_i$ is {\em repelling}.
It is a crucial fact that if the Conley index of $\cN_i$ is nontrivial then $\Inv|\cN_i| \neq \emptyset$.


\FloatBarrier
\subsection{Application of the automatic analysis method to the Chialvo model of a neuron}
\label{sec:appl}
\vskip 5mm

By applying the methods introduced in Sections \ref{sec:glob} and~\ref{sec:indiv} to the Chialvo model of a neuron, explained in Section~\ref{sec:model}, we obtained the continuation diagram shown in Figure~\ref{fig:n12cont}.
The computations were restricted to the phase space $(x,y) \in B := [-0.1,9] \times [-5,3]$ and the parameter set $(b,k) \in \Lambda_2 := [0,1] \times [0.015,0.030]$ with $a := 0.89$ and $c := 0.28$ fixed.
The $1024 \times 1024$ uniform rectangular grid was applied in $B$, and $\Lambda_2$ ws split into $200 \times 75$ rectangles of equal size.
The technical details and justification of these choices are gathered in Appendix~\ref{app:continuation}. Here we only briefly mention that these sets were chosen on the basis of the information contained in~\cite{Chialvo} and our preliminary computations, including application of the topological-numerical analysis with a larger set of parameters $(b,k) \in \Lambda_1 := [0,1] \times [0,0.2]$. The results of the latter computations are shown in Figure~\ref{fig:n09cont} in Appendix~\ref{app:continuation} and are also available on-line at~\cite{www}.

\begin{figure}[htbp]
\centerline{\includegraphics[width=0.8\textwidth]{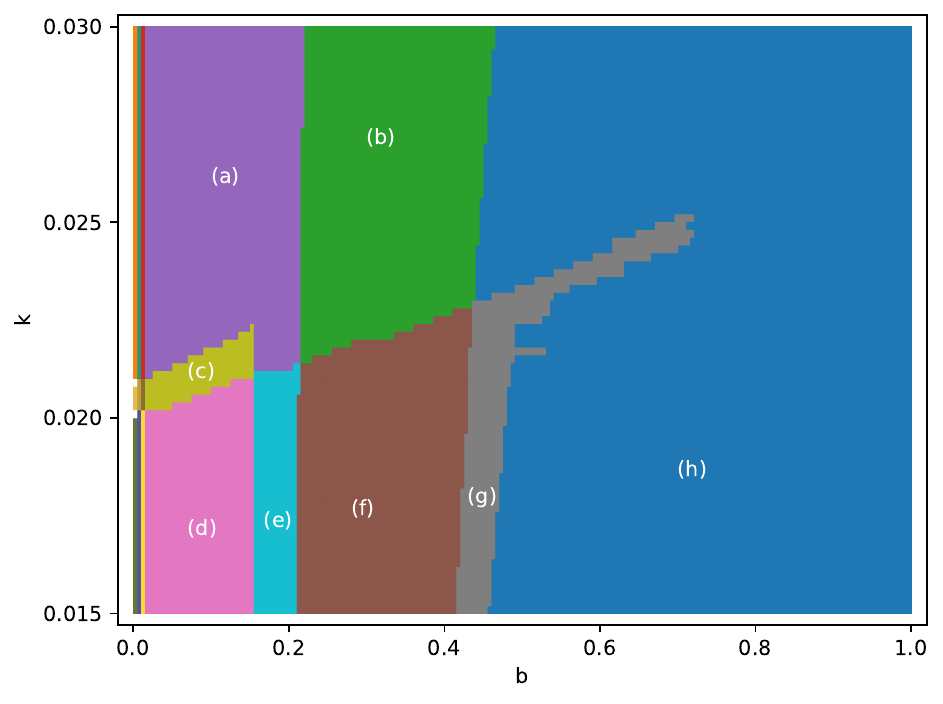}}
\caption{\label{fig:n12cont}%
Continuation diagram for the Chialvo model with $a = 0.89$, $c = 0.28$, and $(b,k) \in \Lambda_2 = [0,1] \times [0.015,0.030]$ split into the $200 \times 75$ uniform rectangular grid. See also Figure~\ref{fig:n09cont} in Appendix~\ref{app:continuation} for a corresponding diagram with $(b,k) \in \Lambda_1 = [0,1] \times [0,0.2] \supset \Lambda_2$, in which one can see that in fact regions (a) and (h) are related by continuation.}
\end{figure}

\label{rev:contExplain}
The continuation diagram in Figure~\ref{fig:n12cont} shows the set of parameters $(b,k) \in \Lambda_2 = [0,1] \times [0.015,0.030]$ split into $200 \times 75$ rectangular boxes of the same size. Each box is thus a subset of parameters; for example, the leftmost bottom box corresponds to $(b,k) \in [0,0.005] \times [0.015,0.0152]$. Adjacent boxes that are shown in the same color belong to the same continuation class (rigorously validated, as explained in Section~\ref{sec:glob}). This means that the dynamics for all the parameters in a common contiguous color area in the diagram is the same from the qualitative point of view, as perceived at the given resolution $1024 \times 1024$ in the phase space. In particular, the number of recurrent components (numerical Morse sets) found for all the parameters in that area, as well as their stability type (measured by the Conley index) are the same.

The continuation diagrams shown in Figures~\ref{fig:n12cont} and~\ref{fig:n09cont} (the latter in Appendix~\ref{app:continuation}) are available in~\cite{www} for interactive browsing. Clicking a point in the continuation diagram launches a page with the phase space portrait of the numerical Morse decomposition computed for the specific rectangle of parameters, as well as a visualization of the corresponding Conley-Morse graph. The details shown in the visualization are briefly explained in Figures \ref{fig:mdec12demo} and~\ref{fig:mgraph12demo}.

\begin{figure}[htbp]
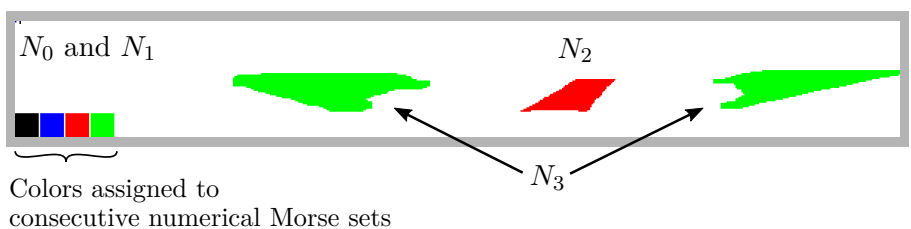

\setlength{\unitlength}{12cm}%
\begin{picture}(1,0.24794625)%
\small
\put(0,0){\includegraphics[width=\unitlength,page=1]{\detokenize{mdec12p2_14explain}}}%
\put(0.01392368,0.19826713){$N_0$ and $N_1$}%
\put(0,0){\includegraphics[width=\unitlength,page=2]{\detokenize{mdec12p2_14explain}}}%
\put(0.6108589,0.19606656){$N_2$}%
\put(0.58,0.0550039){$N_3$}%
\put(0,0){\includegraphics[width=\unitlength,page=3]{\detokenize{mdec12p2_14explain}}}%
\footnotesize
\put(0.00322654,0.04166339){Colors assigned to}%
\put(0.00322654,0.01){consecutive numerical Morse sets}%
\end{picture}%
\caption{\label{fig:mdec12demo}%
Sample phase space portrait computed for the Chialvo model, as shown in the interactive visualization available at~\cite{www},
with the gray bounding box added for clarity.
The color boxes in the lower left corner indicate the color coding of consecutive numerical Morse sets.
The numbering starts with $0$.
Note the barely visible tiny sets $N_0$ (black) and $N_1$ (blue), both located in the left top corner of the figure.}
\end{figure}

\begin{figure}[htbp]
\setlength{\unitlength}{6cm}%
\begin{picture}(2,1)%
\put(0,0){\includegraphics[height=\unitlength]{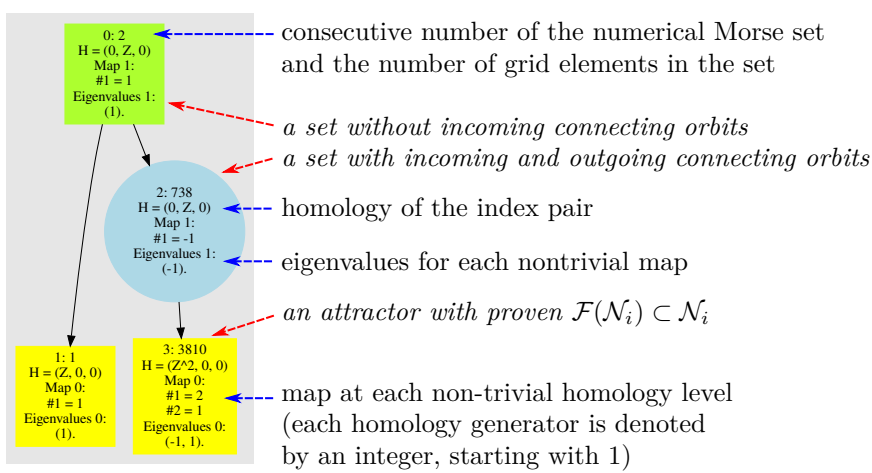}}%
\footnotesize
\put(0.61,0.94){consecutive number of the numerical Morse set}%
\put(0.61,0.87){and the number of grid elements in the set}%
\put(0.61,0.73){\textit{a set without incoming connecting orbits}}%
\put(0.61,0.66){\textit{a set with incoming and outgoing connecting orbits}}%
\put(0.61,0.55){homology of the index pair}%
\put(0.61,0.43){eigenvalues for each nontrivial map}%
\put(0.61,0.32){\textit{an attractor with proven $\cF(\cN_i) \subset \cN_i$}}%
\put(0.61,0.14){map at each non-trivial homology level}%
\put(0.61,0.07){(each homology generator is denoted}%
\put(0.61,0){by an integer, starting with 1)}%
\end{picture}%
\caption{\label{fig:mgraph12demo}%
Sample Conley-Morse graph computed for the Chialvo model, as shown in the interactive visualization available at~\cite{www},
with the gray background added for clarity.
The information in the boxes and ovals corresponds to the information shown in Figures \ref{fig:indStable}--\ref{fig:indTrivial}; in particular, one can determine the stability type of each numerical Morse set by comparison with those examples.
Stable (attracting) numerical Morse sets are indicated by yellow rectangles (it has been proved that $\cF(\cN_i) \subset \cN_i$).
Numerical Morse sets that have no trajectories coming from other Morse sets in the decomposition are indicated
by green rectangles.
Pass-through sets (with orbits coming in from other numerical Morse sets and orbits leaving towards other numerical Morse sets) are indicated by ovals.}
\end{figure}


\FloatBarrier
\subsection{Types of dynamics and bifurcations found in the system}
\label{sec:bifurcations}

A comprehensive overview of the types of dynamics that were found in the model can be seen in Figure~\ref{fig:n12graphs}. For each continuation class, a simplified Conley-Morse graph is shown. An overview of bifurcations that were detected in the system, on the other hand, is better visible in Figure~\ref{fig:n12bif}, where the Conley-Morse graphs were joined by edges whenever the corresponding parameter regions were adjacent (this adjacency can be seen in Figures \ref{fig:n12cont} and~\ref{fig:n12graphs}).
Let us now discuss the Morse decompositions found in the different continuation classes shown in Figure~\ref{fig:n12cont}, and also the bifurcations that were observed.

\begin{figure}[htbp]
\centerline{\includegraphics[width=0.7\textwidth]{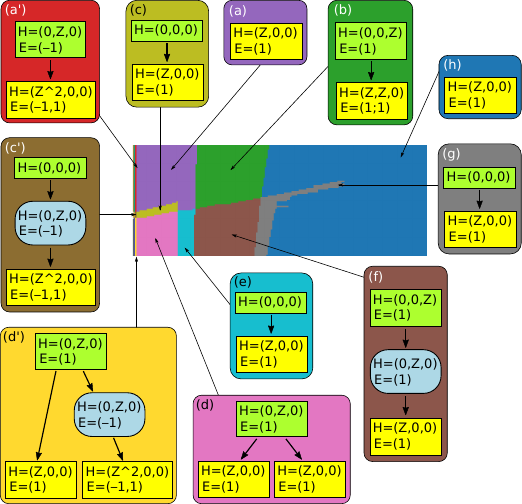}}
\caption{\label{fig:n12graphs}%
Conley-Morse graphs for the continuation classes shown in Figure~\ref{fig:n12cont} for the Chialvo model with $a = 0.89$, $c = 0.28$, and $(b,k) \in \Lambda_2 = [0,1] \times [0.015,0.030]$. The background color of each frame is the same as that of the corresponding continuation class in Figure~\ref{fig:n12cont}.}
\end{figure}

\begin{figure}[htbp]
\centerline{\includegraphics[width=\textwidth]{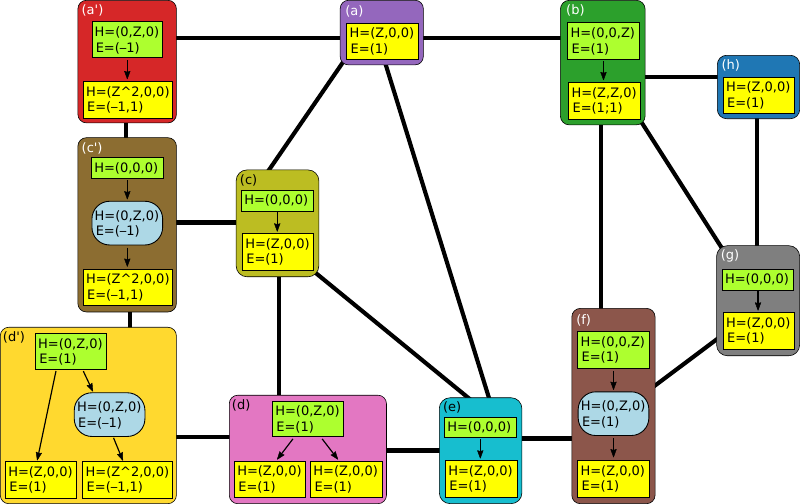}}
\caption{\label{fig:n12bif}%
Continuation graph between the continuation classes shown with the corresponding Conley-Morse graphs for the Chialvo model with $a = 0.89$, $c = 0.28$, and $(b,k) \in \Lambda_2 = [0,1] \times [0.015,0.030]$. The background color of each frame is the same as that of the corresponding continuation class in Figure~\ref{fig:n12cont}. Connections corresponding to classes intersecting by a single vertex, such as between (a) and (f), are neither shown here nor discussed in the text for the sake of clarity.}
\end{figure}

In Region~(a), there is exactly one attracting neighborhood, so the detected dynamics is very simple. However, the constructed numerical Morse set is of different size and shape, depending on the specific part of the region: it is small for lower values of $b$, and suddenly increases in size above Region~(e), as shown in Figure~\ref{fig:mdec12a}.

\begin{figure}[htbp]
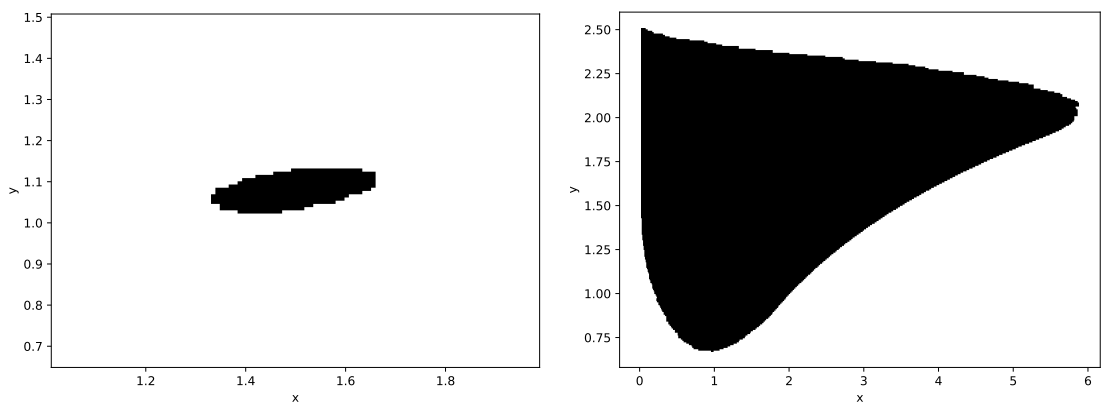

\centerline{\includegraphics[width=0.48\textwidth]{\detokenize{mdec12p21_56}}
\includegraphics[width=0.48\textwidth]{\detokenize{mdec12p36_56}}}
\caption{\label{fig:mdec12a}%
Morse decompositions for selected parameter boxes in Region~(a) shown in Figure~\ref{fig:n12cont}. The parameter boxes have integer coordinates $(21,56)$ and $(36,56)$, and correspond to $(b,k) \in [0.105,0.110] \times [0.0262,0.0264]$, and $(b,k) \in [0.180,0.185] \times [0.0262,0.0264]$, respectively. There is exactly one attracting numerical Morse set in each case. Note the different scale in both plots.}
\end{figure}

When the parameter $b$ is increased to move from Region~(a) to Region~(b), the internal structure of the large isolating neighborhood is revealed, and in Region~(b) one can see it split into an attractor--repeller pair: a small repeller ($308$ boxes) surrounded by a circle-shaped attractor (almost $31{,}000$ boxes); see Figure~\ref{fig:mdec12b}. The Morse graph shows the Conley indices computed for the numerical Morse sets. The exit set of the small set surrounds it: the relative homology is like for the pointed sphere, with the identity index map. The exit set of the large set is empty, and the index map shows that the orientation is preserved. This situation corresponds to what is shown in Figure~\ref{fig:indUnstable}~(c).

\begin{figure}[htbp]
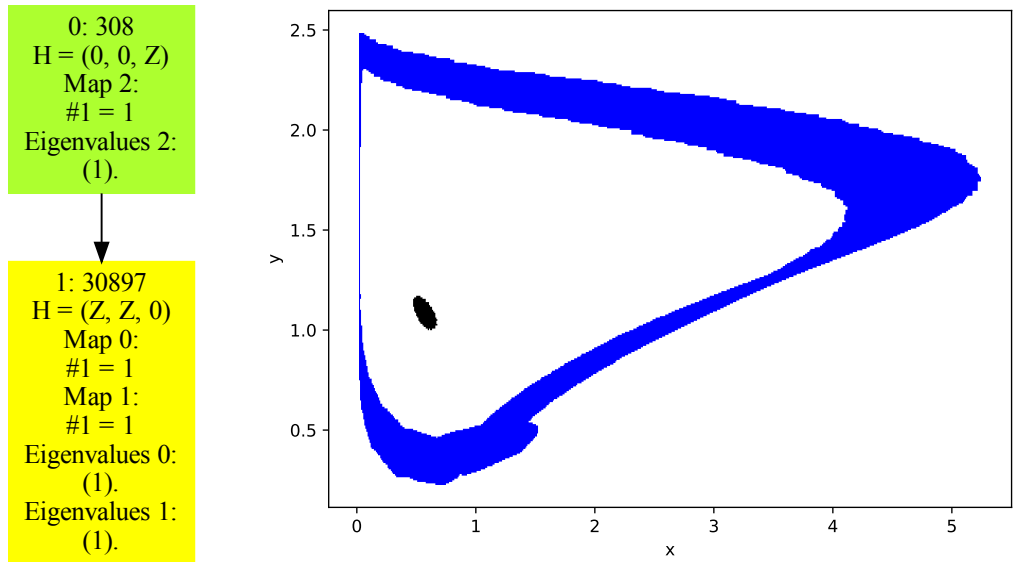

\centerline{\includegraphics[height=0.5\textwidth]{\detokenize{mgraph12p56_56}}
\hspace{12pt}
\includegraphics[height=0.5\textwidth]{\detokenize{mdec12p56_56}}}
\caption{\label{fig:mdec12b}%
The Morse graph (left) and the Morse decomposition (right) computed for a sample parameter box taken from Region~(b) shown in Figure~\ref{fig:n12cont}. The parameter box with integer coordinates $(56,56)$ was chosen, which corresponds to $(b,k) \in [0.280,0.285] \times [0.0262,0.0264]$. There is an attractor in the shape of a circle, and a small repeller inside.}
\end{figure}

On the other hand, when we move from Region~(a) through Region~(c) down to Region~(d) by decreasing the parameter $k$, we observe a numerical version of the saddle-node bifurcation. A new numerical Morse set appears in Region~(c) with trivial index, which then splits into two Morse sets: one with one unstable direction (a saddle) and one attractor. These features can be derived from the Conley index; see Figure~\ref{fig:mdec12d} and compare the indices to the ones shown in Figures \ref{fig:indUnstable}~(a) and~\ref{fig:indStable}~(a).

\begin{figure}[htbp]
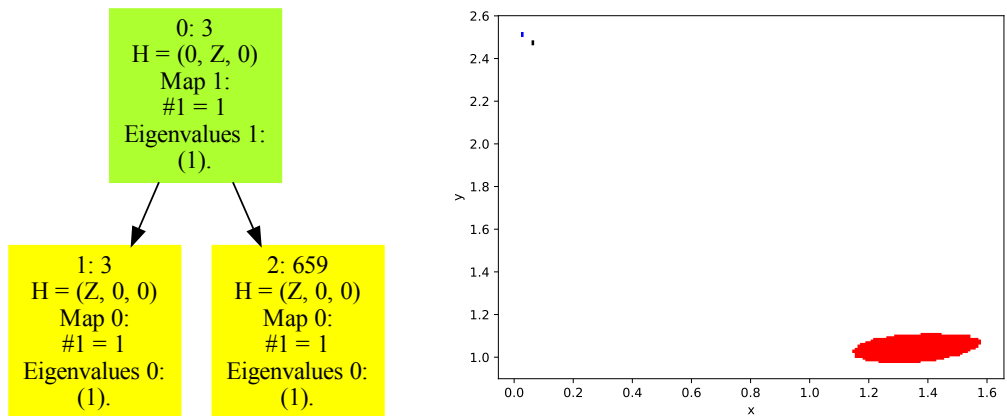

\centerline{\includegraphics[height=0.37\textwidth]{\detokenize{mgraph12p24_16}}
\hspace{12pt}
\includegraphics[height=0.37\textwidth]{\detokenize{mdec12p24_16}}}
\caption{\label{fig:mdec12d}%
The Morse graph (left) and the Morse decomposition (right) computed for a sample parameter box taken from Region~(d) shown in Figure~\ref{fig:n12cont}. The parameter box with integer coordinates $(24,16)$ was chosen, which corresponds to $(b,k) \in [0.120,0.125] \times [0.0182,0.0184]$. There is a small attractor--repeller pair and another attractor; this is a case of bi-stability.}
\end{figure}

\begin{figure}[htbp]
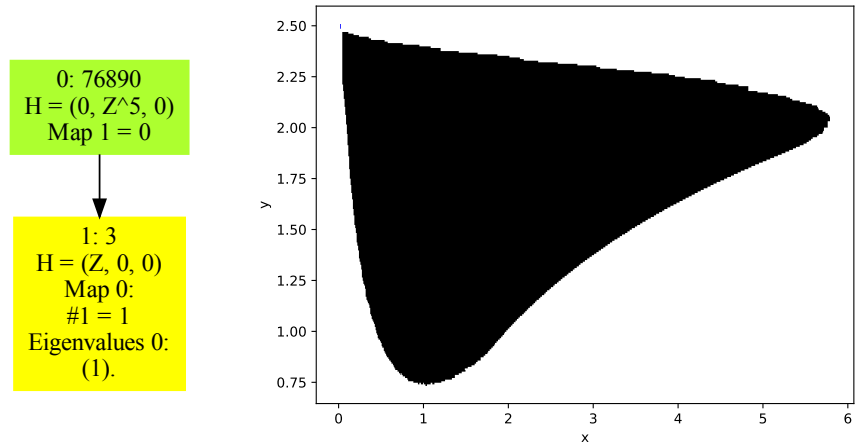

\centerline{\raisebox{0.05\textwidth}{\includegraphics[height=0.3\textwidth]{\detokenize{mgraph12p35_23}}}
\hspace{12pt}
\includegraphics[height=0.4\textwidth]{\detokenize{mdec12p35_23}}}
\caption{\label{fig:mdec12e}%
The Morse graph (left) and the Morse decomposition (right) computed for a sample parameter box taken from Region~(e) shown in Figure~\ref{fig:n12cont}. The parameter box with integer coordinates $(35,23)$ was chosen, which corresponds to $(b,k) \in [0.175,0.180] \times [0.0196,0.0198]$. The Conley index of the large numerical Morse set is trivial. The tiny attractor is barely visible at the top left corner of the figure.}
\end{figure}

\begin{figure}[htbp]
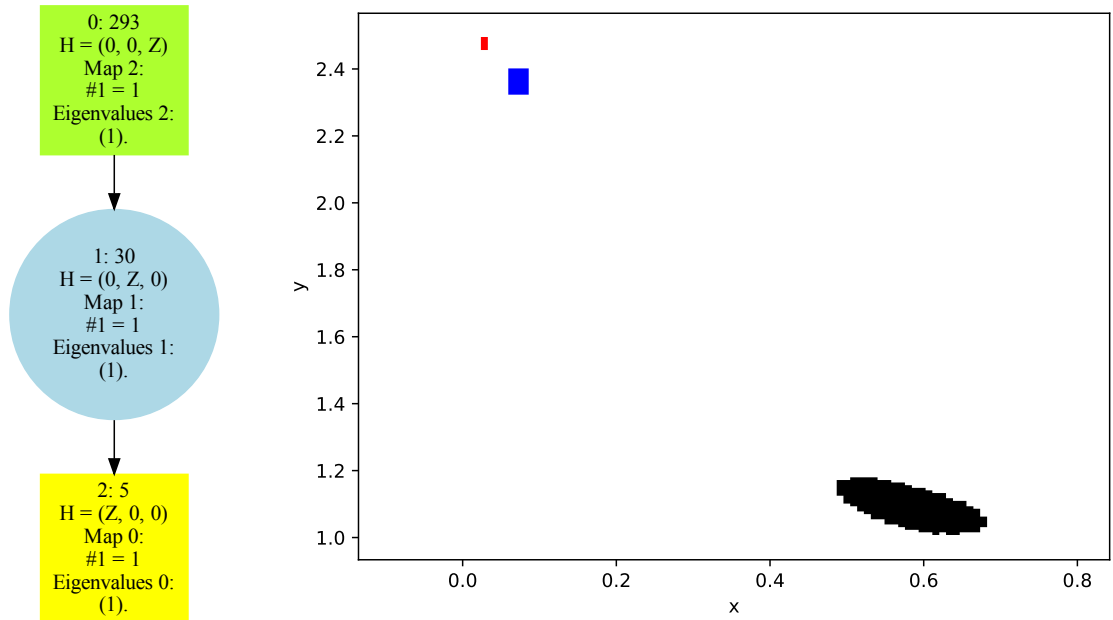

\centerline{\includegraphics[height=0.55\textwidth]{\detokenize{mgraph12p55_23}}
\hspace{12pt}
\includegraphics[height=0.55\textwidth]{\detokenize{mdec12p55_23}}}
\caption{\label{fig:mdec12f}%
The Morse graph (left) and the Morse decomposition (right) computed for a sample parameter box taken from Region~(f) shown in Figure~\ref{fig:n12cont}. The parameter box with integer coordinates $(55,23)$ was chosen, which corresponds to $(b,k) \in [0.275,0.280] \times [0.0196,0.0198]$. There is a repeller (shown in black), a saddle (shown in blue) and a small attractor (shown in red).}
\end{figure}

When the parameter $b$ is increased to move from Region~(d) to Region~(e), the newly created saddle joins the large attractor, and a large numerical Morse set appears; see Figure~\ref{fig:mdec12e}. The Conley index of this large numerical Morse set, however, is trivial, which suggests that it might contain no non-empty invariant set. Its existence is most likely due to the dynamics slowing down in preparation for another bifurcation. Such a bifurcation indeed appears if we increase $b$ further to enter Region~(f). The large numerical Morse set splits into a saddle and a repeller, which is another version of the saddle-node bifurcation; see Figure~\ref{fig:mdec12f} and compare the indices to the ones shown in Figures \ref{fig:indUnstable}~(a) and~\ref{fig:indUnstable}~(c). Increasing the parameter $b$ further makes these two sets collapse in Region~(g) and disappear in Region~(h).

An interesting and somewhat unusual bifurcation occurs when one decreases the parameter $k$ to move from Region~(b) to Region~(f). The circle-shaped attractor observed in Region~(b) splits into a node-type attractor and a saddle, while the node-type repeller inside persists. Apparently, this might be a saddle--node bifurcation. Right after the transition, a small neighborhood of the attractor appears close to the large circular isolating neighborhood of the saddle, and the latter one suddenly shrinks with further decrease in $b$; see Figure~\ref{fig:mdec12bf}.

\begin{figure}[htbp]
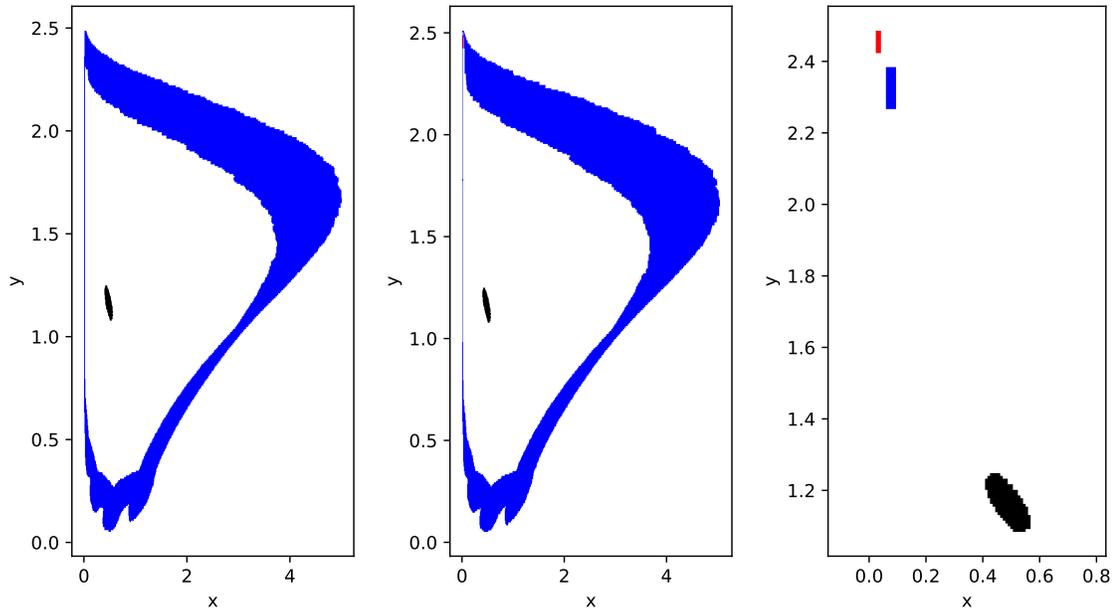

\centerline{\includegraphics[height=0.55\textwidth]{\detokenize{mdec12p62_35}}
\hspace{-6pt}
\includegraphics[height=0.55\textwidth]{\detokenize{mdec12p62_30}}
\hspace{-6pt}
\includegraphics[height=0.55\textwidth]{\detokenize{mdec12p62_29}}}
\caption{\label{fig:mdec12bf}%
Numerical Morse decompositions in transition from Region~(b) to Region~(f) shown in Figure~\ref{fig:n12cont}. The leftmost plot corresponds to the parameter box with integer coordinates $(62,35)$, locatd in Region~(b), the other two -- parameter boxes $(62,30)$ and $(62,29)$, respectively, both located in Region~(f). The actual parameters are: $b \in [0.310,0.315]$ in all the three cases, and $k \in [0.0208,0.0210]$, $k \in [0.0220,0.0212]$, or $k \in [0.0210,0.0212]$, respectively. In the middle plot, the attractor shown in red is barely visible at the top left corner, surrounded by the blue set. Note that the blue set is topologically a circle, but its leftmost edge is very thin and might not be well visible in the plots.}
\end{figure}

There are also a few additional regions in the continuation diagram for very small values of the parameter $b$ that can be seen in Figure~\ref{fig:n12cont} and can be investigated with the interactive continuation diagram available at~\cite{www}. When decreasing $b$ from Regions (a), (c) and (d) to Regions (a$'$), (c$'$) and (d$'$), respectively, that is to $b \in [0.010,0.015]$, an attracting isolating neighborhood splits into a period-two attracting orbit and a saddle in the middle, with the map reversing the orientation, like in a typical period-doubling bifurcation; compare the indices to the ones shown in Figures \ref{fig:indStable}~(b) and~\ref{fig:indUnstable}~(b). When $b$ is decreased even further, at $b \in [0.005,0.010]$, the period-doubling bifurcation is undone, and the two numerical Morse sets again become one. For the lowest values of $b$, that is, $b \in [0,0.005]$, an additional numerical Morse set appears that looks like a layer on top of the attracting numerical Morse set. Its Conley index is trivial, and thus its appearance is most likely due to slow-down in the dynamics; see Figure~\ref{fig:mdec12smallb}.

\begin{figure}[htbp]
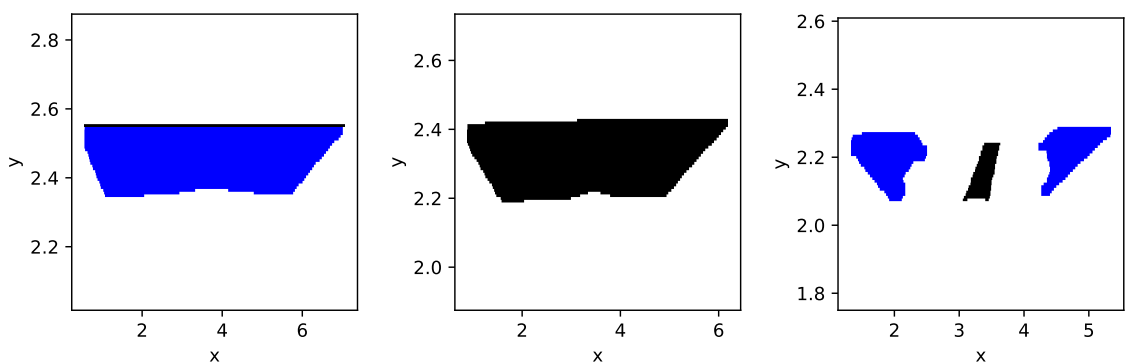

\centerline{\includegraphics[width=0.33\textwidth]{\detokenize{mdec12p0_51}}
\hspace{-6pt}
\includegraphics[width=0.33\textwidth]{\detokenize{mdec12p1_51}}
\hspace{-6pt}
\includegraphics[width=0.33\textwidth]{\detokenize{mdec12p2_51}}}
\caption{\label{fig:mdec12smallb}%
Numerical Morse decompositions for very small values of $b$ shown in Figure~\ref{fig:n12cont} as vertical stripes along the left-hand side edge of the diagram. The parameter boxes chosen for the plots are (from left to right): $(0,51)$, $(1,51)$, and $(2,51)$. The actual parameters are: $k \in [0.0252,0.0254]$ in all the three cases, and $b \in [0,0.005]$, $b \in [0.005,0.01]$, or $b \in [0.01,0.015]$, respectively. The thin numerical Morse set (drawn in black) in the leftmost plot has trivial Conley index. The isolating neighborhood in the middle plot is an attractor. The Morse decomposition in the rightmost plot looks like a period-two attracting orbit with a saddle in the middle.}
\end{figure}

We would like to point out the fact that our discussion of the dynamics and bifurcations was only based on isolating neighborhoods and their Conley indices. The actual dynamics might be much more subtle and complicated, and therefore, any statements about possible hyperbolic fixed points are merely speculations. Moreover, the actual bifurcations may take place for some nearby parameters, at locations that are somewhat shifted from the lines shown in Figures \ref{fig:n09cont} and~\ref{fig:n12cont}. Nevertheless, if the isolating neighborhoods are small then, from the point of view of applications in which the accuracy is limited and there is some noise or other disturbances, the numerical results shed light onto the global dynamics and our discussion explains it in terms of simple models that can be built with hyperbolic fixed points. This is not true, however, in the cases in which the computations yield large isolating neighborhoods. Indeed, the computational method introduced in \cite{Arai} does not provide any means for understanding the dynamics inside such sets, apart from what can be deduced from the knowledge of their Conley indices. We proposed some methods for this purpose in Section~\ref{sec:indiv} and we show their application in Section~\ref{sec:recComp}.


\FloatBarrier
\subsection{Sizes of invariant sets in the Chialvo model and the spiking-bursting oscillations}
\label{sec:sizes}

The diagram in Figure~\ref{fig:sizeNeuron12} shows the total size of all the numerical Morse sets constructed for all the parameter combinations considered. This diagram complements the corresponding continuation diagram (Figure~\ref{fig:n12cont}) in providing the information about the global overview of the dynamics. Note that a corresponding diagram was also computed for the wider ranges of the two parameters, which provides a more suggestive picture; see Figure~\ref{fig:sizeNeuron09} in Appendix~\ref{sec:sizeNeuron09}.

A larger numerical Morse set allows more room for fluctuations or even appearance of complicated dynamics in a real system that is approximated by the mathematical model. In particular, in the Chialvo model, the appearance of spiking-bursting oscillations is connected with the emergence of large numerical Morse sets, especially if this phenomenon is combined with chaotic dynamics. Indeed, in neuron models, such as the Chialvo model, an  attracting equilibrium point corresponds to the resting state of the neuron, whereas tonic (sustained) spiking is connected with the existence of oscillatory solutions (which do not converge to the stable periodic fixed point). However, also the amplitude of these oscillatory solutions must be large enough since oscillatory solutions with small amplitudes would rather correspond to ``subthreshold'' oscillations than to spikes. If such an attracting oscillatory orbit is periodic then the spike-pattern fired by the neuron is (asymptotically) periodic as well. On the other hand, non-periodic oscillatory solutions (such as, for example, the blue orbit shown in Figure~\ref{fig:traj4}) lead to chaotic spiking patterns where irregular spikes with varying amplitudes are observed, often interspersed with small subthreshold oscillations. If some spikes on these orbits are separated only by short interspike intervals, followed by periods of quiescence (no spikes), then we can say that spikes are grouped into bursts, i.e., we have spiking-bursting solutions; this happens typically when the oscillatory solution winds many times around the unstable fixed point (which is the case of the blue orbit in Figure~\ref{fig:traj4}, see also Figure~10 in~\cite{Chialvo}). Therefore, the existence of spiking-bursting solutions is directly connected with the existence of large Morse sets, and our results provided in this work allow to indicate various regions in the parameter space in which one can look for such phenomena. On the other hand, the phenomenon of multistability, such as co-existence of an attracting oscillatory solution and a quiescence solution (attracting fixed point) inside the area delineated by the oscillatory orbit, also often leads to large Morse sets. In such a case a proper perturbation might cause the neuron to switch from sustained periodic firing to resting and vice versa (see also Figure~8 in~\cite{Chialvo}.)

In the Chialvo model, one can notice that very large numerical Morse sets, consisting of some $80{,}000$ grid elements or more, appear especially in two regions of the parameters: $(b,k) \approx (0.5,0.03)$ and $(b,k) \approx (0.2,0.03)$. Such parameters may make the model resistant to purely analytical investigation due to its complexity, making our methods a better fit for the purpose of understanding the dynamics. One may also speculate that in the case of larger sets the dynamics is less predictable and thus chaotic dynamics might emerge.

Although it is not entirely obvious, it may be possible that some numerical Morse decompositions fall in the same continuation class even if the sizes of the numerical Morse sets being matched differ considerably. Indeed, this happens in our case. For example, sudden change in the size of the numerical Morse sets sometimes can be found for $(b,k) \approx (0.8, 0.025)$ or $(b,k) \approx (0.17, 0.026)$. The reason in all the observed cases seems to be the emergence of cyclic behavior that corresponds to the spiking-bursting oscillations in the Chialvo model, as explained above. 

An important observation is that sets of parameters that yield very large numerical Morse sets sometimes span across adjacent continuation classes. This can be observed, for example, for $(b,k) \approx (0.22, 0.025)$, where the Conley-Morse diagram changes only because the unstable fixed point can be isolated from the large attracting isolating neighborhood. This results in a qualitative change in the perception of the dynamics at the prescribed resolution, switching from Region~(a) to Region~(b), even though the actual behavior of the vast majority of trajectories might still be similar.

\begin{figure}[htbp]
\centering
\includegraphics[width=0.9\textwidth]{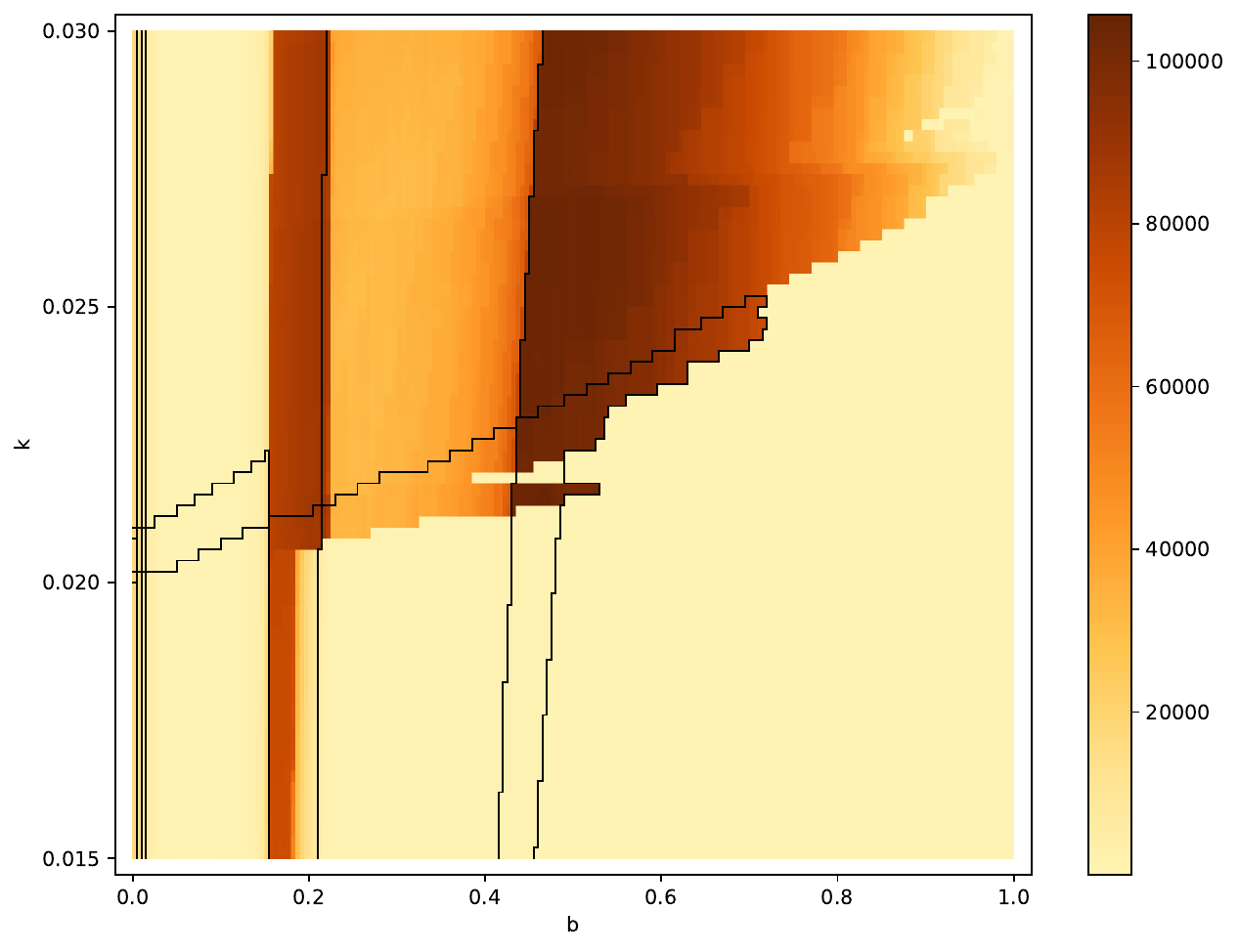}
\caption{\label{fig:sizeNeuron12}%
The size of the union of all the numerical Morse sets found in the phase space for the corresponding parameters in $\Lambda_2$ in the Chialvo model. The black lines indicate borders between different continuation classes; see Figure~\ref{fig:n12cont} for the corresponding continuation diagram.}
\end{figure}


\pagebreak

\section{Finite resolution recurrence and its variation}
\label{sec:recurrence}

While Conley index is a powerful topological tool that provides reliable information about the isolating neighborhood, it \label{sheds} does not provide extra
information about the dynamics inside of the isolating neighborhood. This may be especially disappointing if the neighborhood is large, consisting of hundreds of thousands of grid elements. Since this set corresponds to a strongly connected path component of the graph representation of the map, there exists a path in the graph from every grid element to any other element, including a path back to itself. In particular, a periodic orbit gives rise to a cycle in the graph of the same length, and thus a fixed point (stable or not) yields a cycle of length $1$. On the opposite, a path in the graph corresponds to a pseudo-orbit for the underlying map: after each iteration we may need to switch to another point within the image of the grid element in order to follow a path in the graph by means of pieces of orbits.

We begin by defining the notion of Finite Resolution Recurrence in Section~\ref{sec:frr}, and we show the results of its computation on three dynamical systems of different nature.
Then in Section~\ref{sec:recComp}, we show and discuss the results of computation of Finite Resolution Recurrence (FRR) for some large numerical Morse sets computed for the Chialvo model.
Since we notice that it is the variation of FRR that is crucial in distinguishing between different types of dynamics, we introduce the notion of Finite Resolution Recurrence Variation (FRRV) in Section~\ref{sec:FinVar}, and also its normalized version (NFRRV) in Section~\ref{sec:NormVar}. Finally, in Section~\ref{sec:varComp}, we discuss the computation of both quantities for the six dynamical systems discussed in Sections \ref{sec:frr} and~\ref{sec:recComp}. We also provide a diagram (Figure~\ref{fig:varNeuron}) that shows the result of the computation of NFRRV for a large range of parameters in the Chialvo model.


\FloatBarrier
\subsection{Finite resolution recurrence}
\label{sec:frr}

Recurrence is one of fundamental properties of many dynamical systems. In order to get insight into the recurrent dynamics inside each numerical Morse set, we introduce the notion of Finite Resolution Recurrence (FRR for short) and an algorithm for its analysis. This is a new method for the analysis of dynamics in a Morse set by means of the distribution of minimum return times. For alternative approaches to measure the recurrence in dynamical systems see, for example, \cite{Rplots}.

In what follows, we define a finite resolution version of the notion of recurrence, we explain the rigorous numerical information that it provides about trajectories, and we show the usefulness of this notion on three examples.

\begin{definition}
Let $\cF \colon \cN \multimap \cN$ be a multivalued map on a set $\cN \subset \cG(B)$, and let $Q \in \cN$. The \emph{recurrence time} of $Q$ in $\cN$ with respect to $\cF$ is defined as follows:
$\rec (Q) := \min \{k > 0 : Q \in \cF^k (Q)\}$, with the convention $\min \emptyset = \infty$.
\end{definition}

Recurrence times in a numerical Morse set can be effectively computed. We propose a specific algorithm, prove its correctness, and determine its computational complexity (with proof) in Appendix~\ref{app:recurrence}.

\label{frrRigorous}
At this point we emphasize the fact that knowledge of the recurrence time of a grid element provides certain rigorous information about the dynamics of the points that belong to the grid element. Specifically, if $\rec(Q)=r$ then for every $x \in Q$, we know that if the trajectory $x = x_0, x_1, \ldots, x_{r-1}$ stays in $|\cN|$ then $x_i \notin Q$ for $i = 1, \ldots, r-1$. In particular, there is no periodic orbit contained in $|\cN|$ that goes through $Q$ whose period is below $r$. On the other hand, if $\rec(Q)=r$ then we know that there exists a \emph{$\delta$-pseudo-orbit} $x_0, \ldots, x_r$, with $\delta = \max_{Q \in \cN} \diam (|\cF(Q)|)$, that begins and ends in $Q$; this is a sequence of points $x_i \in |\cN|$ such that $x_0, x_r \in Q$, and $\dist (f(x_{i-1}),x_{i}) < \delta$ for all $i\in \{1,\ldots,r\}$.

In order to illustrate the usefulness of recurrence times, we show three examples. Figure~\ref{fig:recHenon} shows a numerical Morse set constructed for the well-known $2$-dimensional H\'{e}non map \cite{Henon} that exhibits chaotic dynamics. The recurrence times are low, and different values are scattered unevenly throughout the entire set.
It seems that some orbits with low periods were \label{identified} identified correctly, especially the fixed point and the period-two orbit.
Figure~\ref{fig:recLeslie} shows an example of a large numerical Morse set computed for the non-linear Leslie population model discussed in~\cite{Arai}. The recurrence times reveal its internal structure. Indeed, separate isolating neighborhoods for a fixed point in the middle and period-$3$ orbits can be found for nearby parameters or when conducting the computation at a much finer resolution.
Figure~\ref{fig:recVanderpol}, on the other hand, shows a numerical Morse set that is an isolating neighborhood for a time-$t$ discretization of the Van der Pol oscillator flow on $\bR^2$ for the parameters for which the expected attracting periodic trajectory is observed. In this example, one can clearly see the recurrence time $1$ that corresponds to the fixed point in the middle, and high recurrence times around $50$ that identify the stable periodic trajectory.

\begin{figure}[htbp]
\centering
\includegraphics[width=0.6\textwidth]{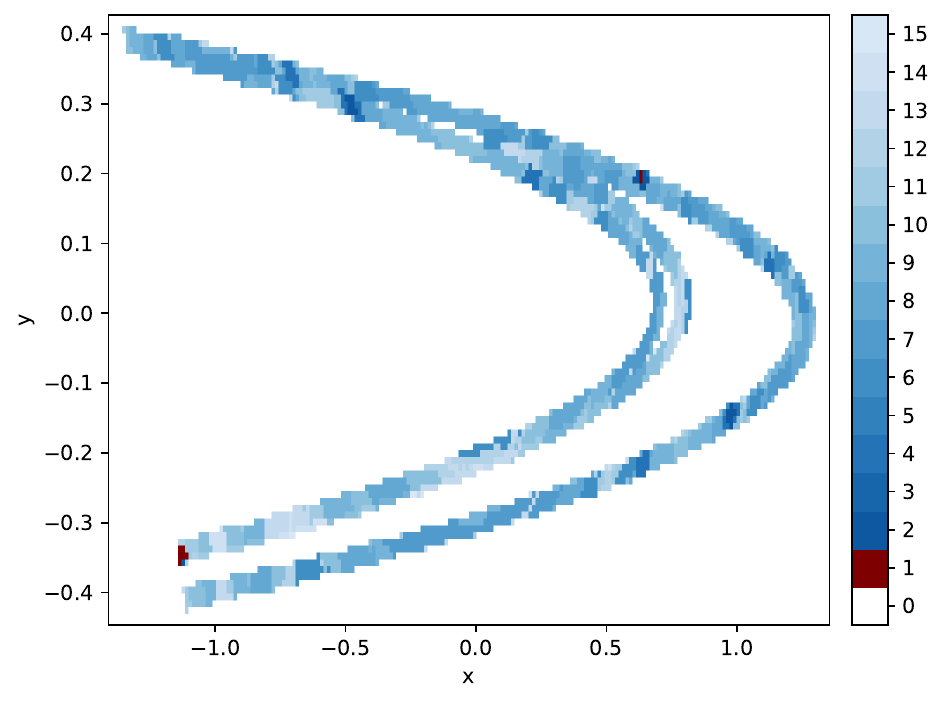} \hfill
\includegraphics[width=0.39\textwidth]{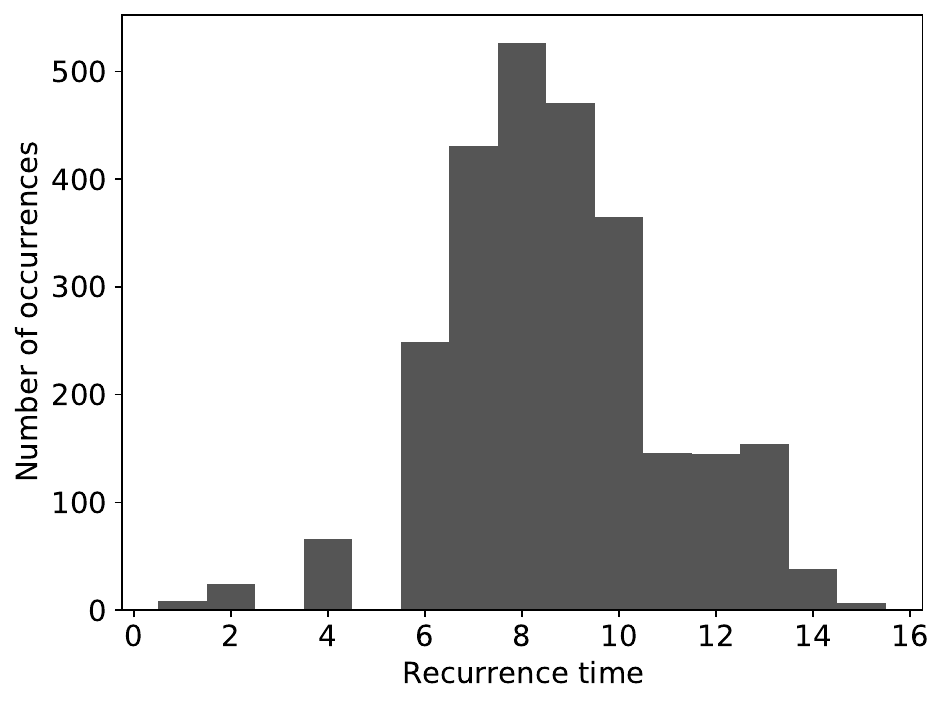}
\caption{\label{fig:recHenon}%
Recurrence diagram for a numerical Morse set constructed for the H\'{e}non map that encloses a chaotic attractor, and a histogram that shows the amounts of the various recurrence times encountered. The fixed points and period-$2$ orbits are clearly visible, as well as the complicated inner structure of the set.}
\end{figure}

\begin{figure}[htbp]
\centering
\includegraphics[width=0.6\textwidth]{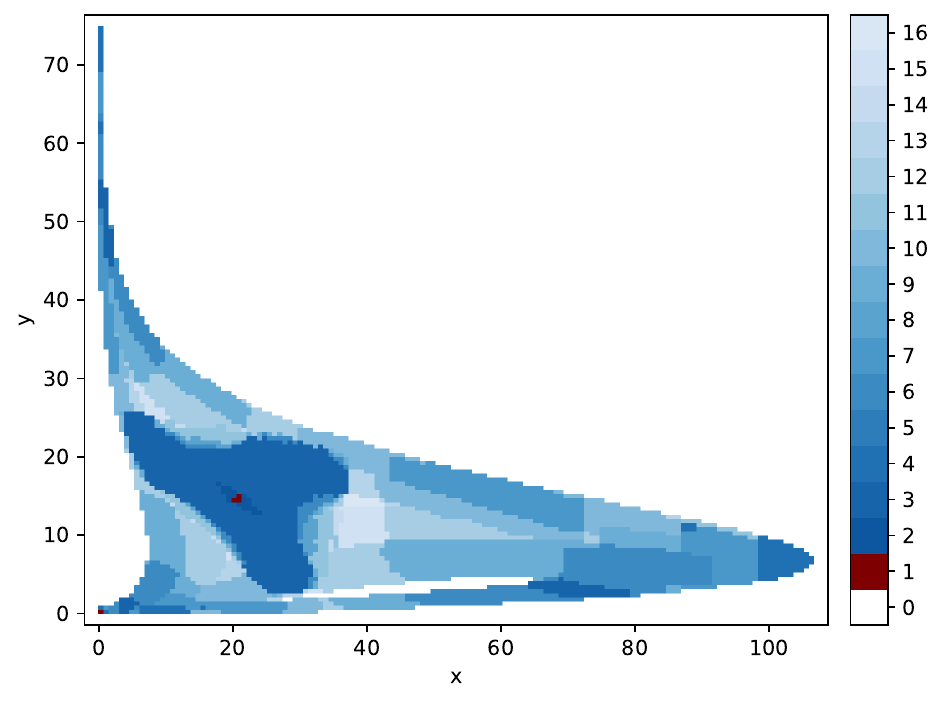} \hfill
\includegraphics[width=0.39\textwidth]{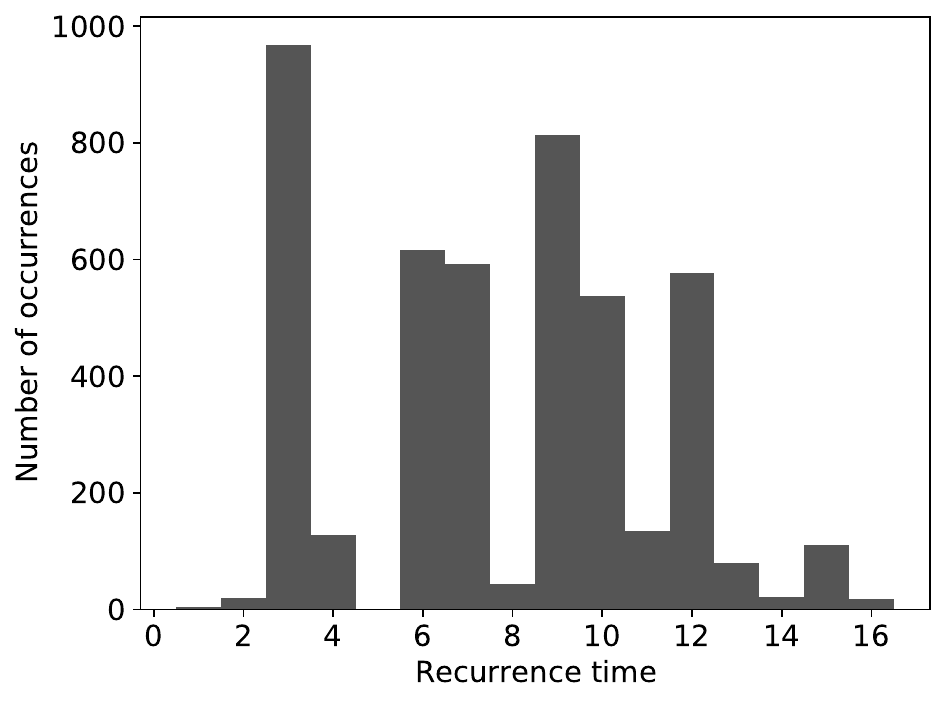}
\caption{\label{fig:recLeslie}%
Recurrence diagram for a numerical Morse set constructed for the Leslie population model map discussed in \cite{Arai}, and a histogram that shows the amounts of the various recurrence times encountered. Inner structure of the large isolating neighborhood is revealed, with a fixed point and period-$3$ orbits.}
\end{figure}

\begin{figure}[htbp]
\centering
\includegraphics[width=0.6\textwidth]{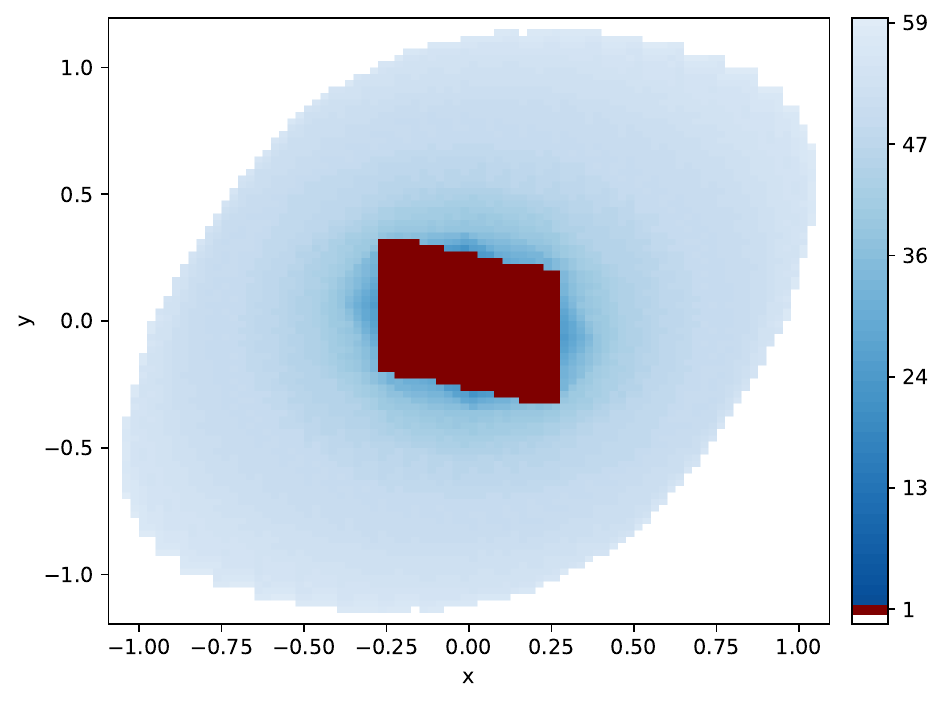} \hfill
\includegraphics[width=0.39\textwidth]{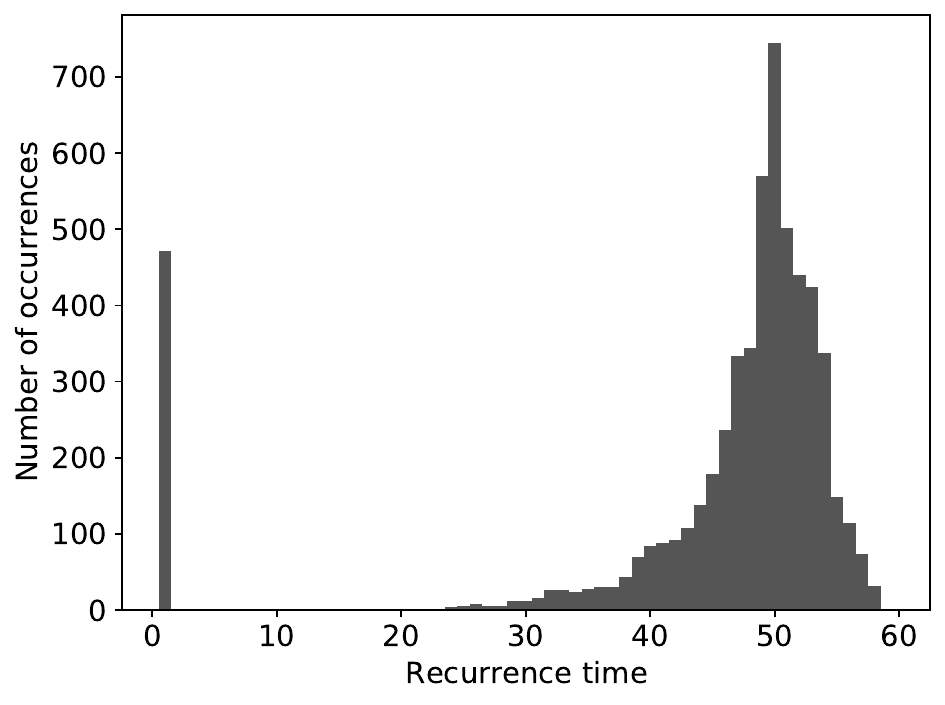}
\caption{\label{fig:recVanderpol}%
Recurrence diagram for a numerical Morse set constructed for a time-$t$ map for the Van der Pol oscillator with a periodic attracting trajectory, and a histogram that shows the amounts of the various recurrence times encountered. The recurrence times for the map vary from $1$ around the fixed point at the origin up to $59$ at the border of the set.}
\end{figure}

Based on the illustrations, one may conjecture that high local variation in the recurrence time is an indicator of complicated dynamics, such as chaos. In Section~\ref{sec:FinVar}, we introduce the notion of variation of the Finite Resolution Recurrence that we further use to effectively quantify this local variation in recurrence time.


\FloatBarrier
\subsection{Recurrence analysis of large invariant sets in the Chialvo model}
\label{sec:recComp}

Through numerical simulations, we identified three kinds of dynamics that yield large numerical Morse sets in the Chialvo model. They are shown in Figure~\ref{fig:traj4},
\label{rev:fourParam}
and were found for some parameter values close to those considered in the paper, and thus not directly related to the continuation diagram shown in Figure~\ref{fig:n12cont}.
The initial condition $(2,1.8)$ was taken in all the cases. Plotting the trajectories for $100{,}000$ iterations was started after $1$ million of initial iterations to allow the trajectories settle down on the attractors; we remark that $10{,}000$ initial iterations were not enough for the winding orbit. In addition to them, an attracting fixed point is shown that appears for another combination of the parameters.

The fixed point around $(0,2.5)$ is shown in Figure~\ref{fig:traj4} in black, the prominent periodic orbit in orange, an orbit that seems to follow a chaotic attractor is shown in partly transparent blue, and a winding periodic orbit with weak attraction is shown in red. This last orbit does not look like part of a chaotic trajectory, because the few dozens points in the figure actually correspond to $100{,}000$ iterates, so this is most likely a periodic attractor. Note the small differences in the bifurcation parameters that yield the qualitatively different asymptotic behavior of the orbits.

\label{rev:numEvidence}
We remark that we did not prove the existence of the attractors nor chaotic dynamics shown in Figure~\ref{fig:traj4}; however, the numerical simulations that we conducted can be treated as strong numerical evidence in favor of such conjectures.

\begin{figure}[htbp]
\centerline{\includegraphics[width=0.8\textwidth]{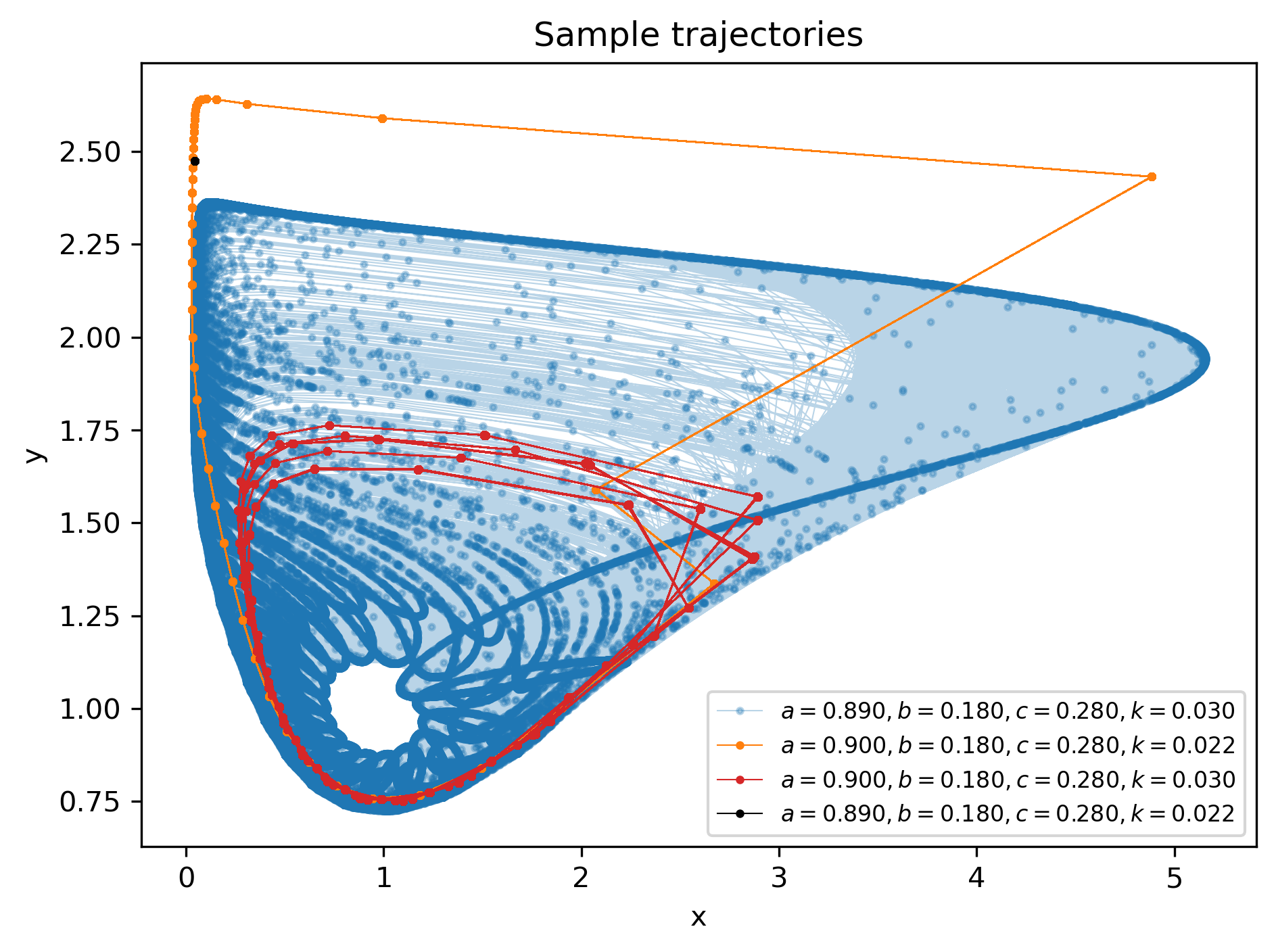}}
\caption{\label{fig:traj4}%
Four specific types of attracting orbits observed in numerical simulations in the Chialvo model, discussed in Section~\ref{sec:recComp}.}
\end{figure}

Recurrence diagrams for the numerical Morse sets constructed for the three combinations of parameters shown in Figure \ref{fig:traj4} that yield large sets are shown in Figures \ref{fig:recNeuron3}--\ref{fig:recNeuron5}. Since the time complexity of the algorithm is worse than $O(|V|^2)$, we conducted the computations at a relatively low resolution in the phase space in order to quickly obtain the results (in less than $2$ minutes each) and to clearly illustrate the results. The constructed numerical Morse sets consist of $8{,}265$, $7{,}740$, and $6{,}740$ grid elements, respectively. Instead of the exact values of the parameters $a$ and $k$, intervals of width $0.0001$ were taken in order to make the computations more realistic.

Recurrence histograms are shown along with each recurrence diagram in Figures \ref{fig:recNeuron3}--\ref{fig:recNeuron5}. There are some differences between these histograms that reflect subtle differences in recurrence diagrams. They are discussed in the next sections.

\begin{figure}[htbp]
\centering
\includegraphics[width=0.6\textwidth]{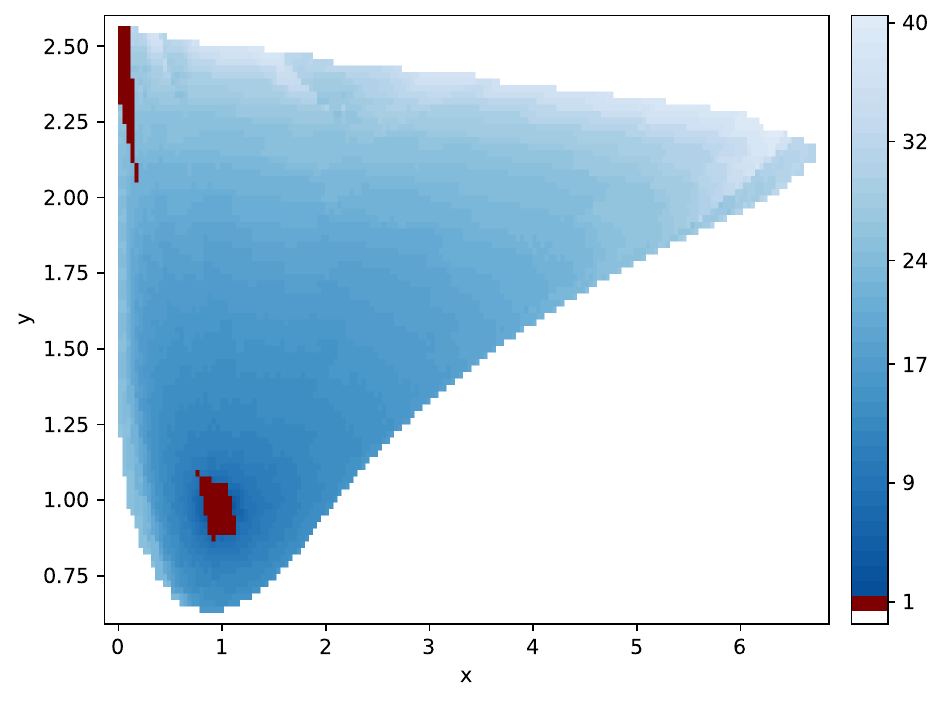}\hfill
\includegraphics[width=0.39\textwidth]{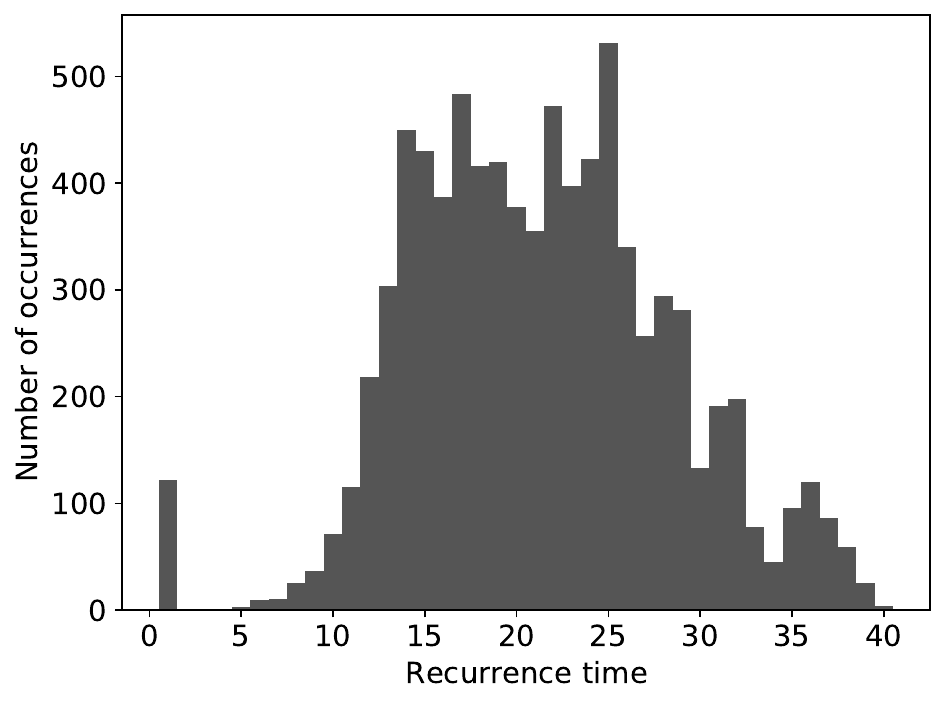}
\caption{\label{fig:recNeuron3}%
Recurrence diagram with histogram computed for the large numerical Morse set found for the Chialvo map with $a\approx 0.89$, $b=0.18$, $c=0.28$, and $k\approx 0.03$, for which $\NFRRV = 1.447$. This is the range of parameters that corresponds to the orbit that looks like a chaotic one in numerical simulations whose results are shown in Figure~\ref{fig:traj4}.}
\end{figure}

\begin{figure}[htbp]
\centering
\includegraphics[width=0.6\textwidth]{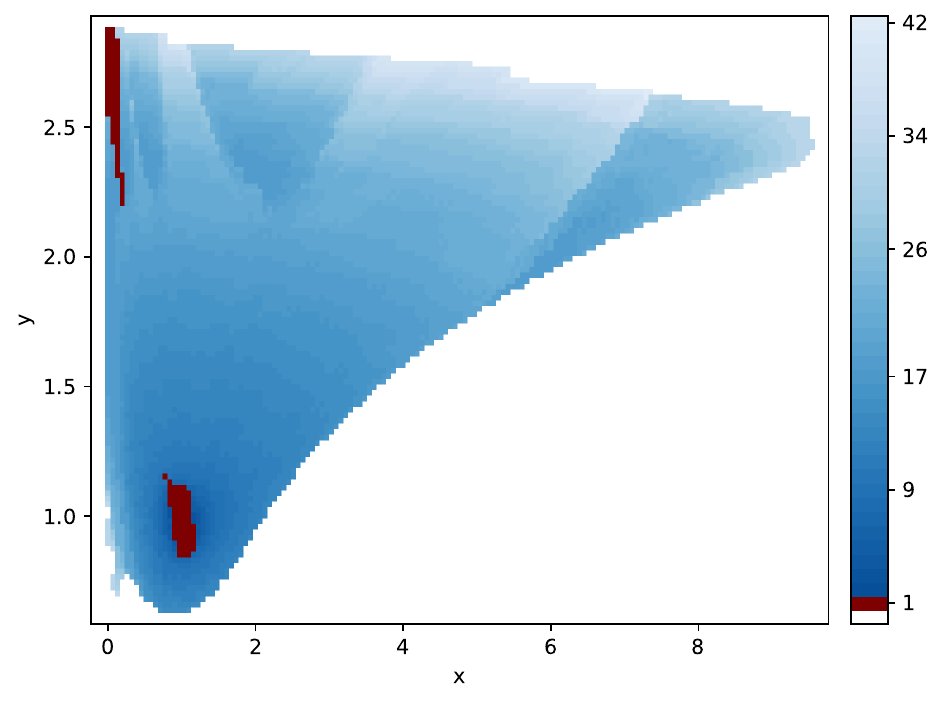}\hfill
\includegraphics[width=0.39\textwidth]{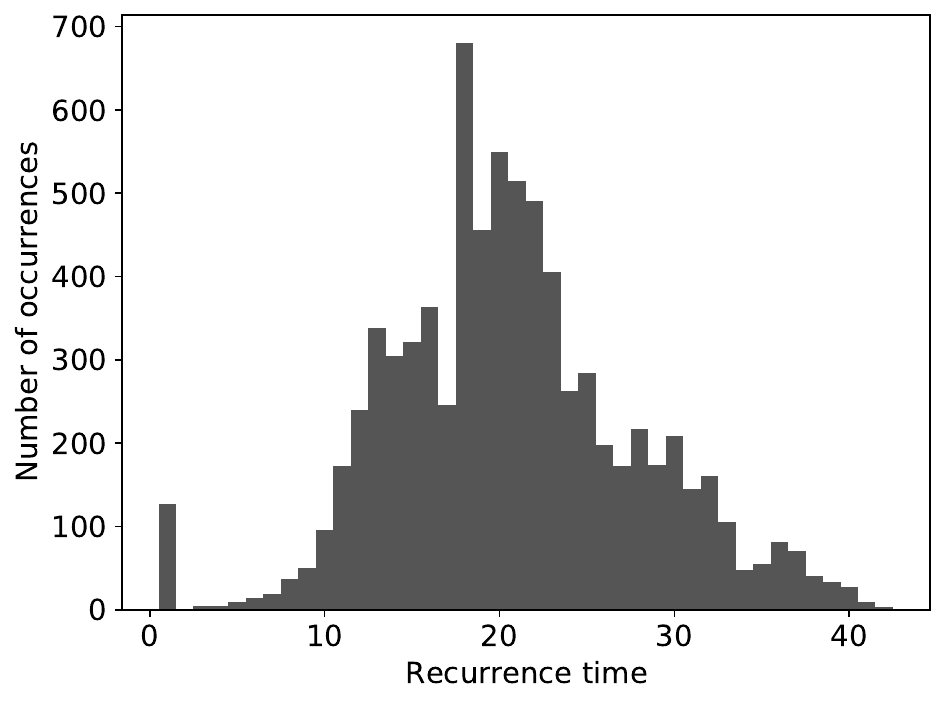}
\caption{\label{fig:recNeuron4}%
Recurrence diagram with histogram computed for the large numerical Morse set found for the Chialvo map with $a\approx 0.9$, $b=0.18$, $c=0.28$, and $k\approx 0.022$, for which $\NFRRV = 1.661$. This is the range of parameters that corresponds to the clear periodic orbit observed in numerical simulations whose results are shown in Figure~\ref{fig:traj4}.}
\end{figure}

\begin{figure}[htbp]
\centering
\includegraphics[width=0.6\textwidth]{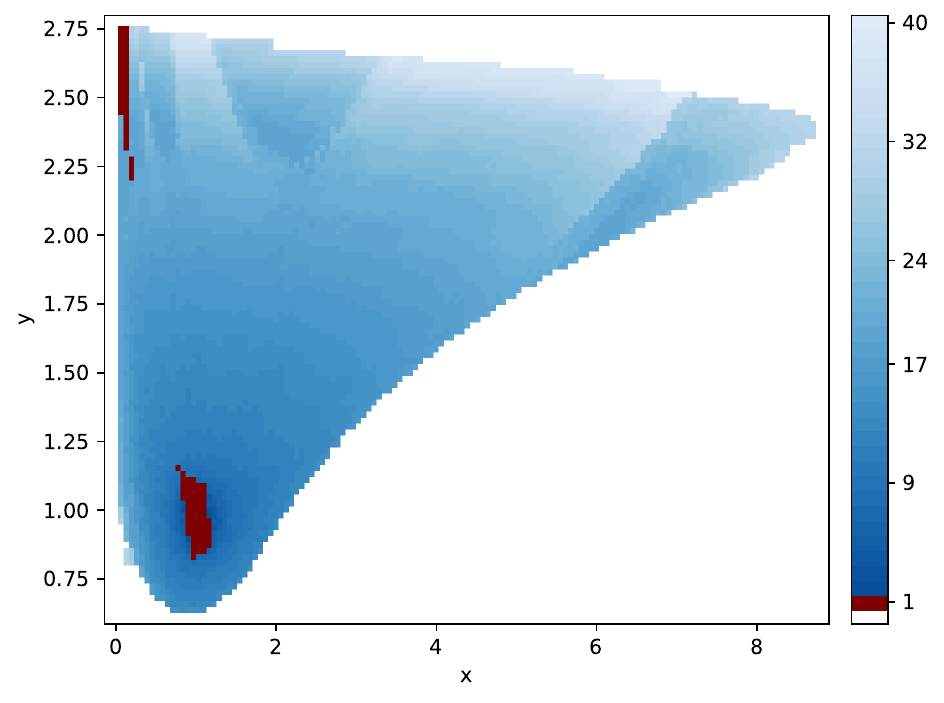}\hfill
\includegraphics[width=0.39\textwidth]{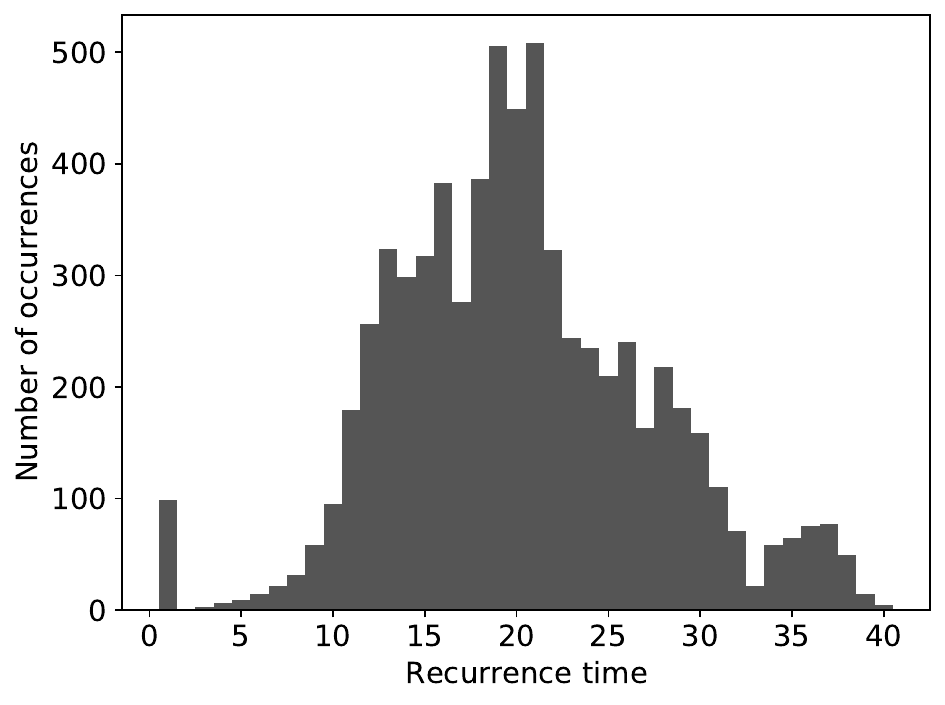}
\caption{\label{fig:recNeuron5}%
Recurrence diagram with histogram computed for the large numerical Morse set found for the Chialvo map with $a\approx 0.9$, $b=0.18$, $c=0.28$, and $k\approx 0.03$, for which $\NFRRV = 1.598$. This is the range of parameters that corresponds to the periodic orbit winding multiple times observed in numerical simulations whose results are shown in Figure~\ref{fig:traj4}.}
\end{figure}


\FloatBarrier
\subsection{Finite Resolution Recurrence Variation}
\label{sec:FinVar}

Let us recall the notion of variation of a function in multiple variables introduced by Vitali \cite{VitaliVar}. It is a generalization of the well known notion of variation of a real-valued function in one variable.

\begin{definition}[see \cite{VitaliVar}]
\label{def:VitaliVar}
Let $f \colon B \to \bR$ be a function defined on a rectangular set $B = [a_1,b_1] \times \cdots \times [a_n,b_n] \subset \bR^n$. For $k \in \{1, \ldots, n\}$ and $h_k > 0$, define
\[
\Delta^{k}_{h_k}(f,x) := f(x_1,\ldots,x_k+h_k,\ldots,x_n) - f(x_1,\ldots,x_k,\ldots,x_n),
\]
and recursively:
\begin{eqnarray*}
\Delta_{h_1} (f,x) & := & \Delta^{1}_{h_1} (f,x), \\
\Delta_{h_1,\ldots,h_k} (f,x) & := & \Delta^{k}_{h_k} \big(\Delta_{h_1,\ldots,h_{k-1}},x\big).
\end{eqnarray*}
\emph{Vitali variation} of $f$ on $B$ is defined as the supremum of the sums
\begin{equation}
\label{VitaliSum}
\sum_{i_1=1}^{N_1} \cdots \sum_{i_n=1}^{N_n} \left|
\Delta_{h_1^{i_1},\ldots,h_n^{i_n}} \big(f,(x_1^{i_1},\ldots,x_n^{i_n})\big)
\right|
\end{equation}
over all the possible finite subdivisions of $[a_1,b_1]$, \ldots, $[a_n,b_n]$, where $h^{i_j}_k$ is the difference between the subdivision points $x_{k}^{1},\ldots,x_{k}^{N_k+1}$ of $[a_k,b_k]$.
\end{definition}

Note that the recurrence time function is constant on the interior of each grid element, and can be set to the minimum of the values of the intersecting grid elements at each boundary point. Therefore, for the practical computation of this variation, we only need to check the differences in the values of the function on adjacent grid elements. In two dimensions, the formula on a grid of $(N+1) \times (M+1)$ rectangular grid elements reduces to the following:
\begin{equation}
\label{Vitali2D}
\sum_{i=1}^{N} \sum_{j=1}^{M} \left| f(x_{i+1},x_{j+1}) - f(x_{i+1},x_j) - f(x_i,x_{j+1}) + f(x_i,x_j) \right|,
\end{equation}
and we consider only those values for which the four grid elements are all in the set that we analyze. We make this more precise in the following.

\begin{definition}[Finite Resolution Recurrence Variation (FRRV)]
\label{def:FinVar}
Let $\cN \subset \cG(B)$ and let $f \colon \cN \to \bR$.
Given an $n$-tuple of integers $(j_1,\ldots,j_n)$, denote by $Q(j_1,\ldots,j_n)$ the grid element in $\cG(B)$ referred to by this $n$-tuple of integers.
Let $I(\cN)$ denote the set of those grid elements $Q(j_1,\ldots,j_n) \in \cN$ for which all the grid elements of the form $Q(j_1+\delta_1,\ldots,j_n+\delta_n)$, where $\delta_i \in \{0,1\}$, are in $\cN$.
Following the idea of Definition~\ref{def:VitaliVar}, define
\[
\Delta_{k}\big(f,Q(j_1,\ldots,j_n)\big) := f\big(Q(j_1,\ldots,j_k+1,\ldots,j_n)\big) - f\big(Q(j_1,\ldots,j_k,\ldots,j_n)\big),
\]
and recursively:
\[
\Delta_{1,\ldots,k} \big(f,Q(j_1,\ldots,j_n)\big) := \Delta_{k} \big(\Delta_{1,\ldots,{k-1}},Q(j_1,\ldots,j_n)\big).
\]
\emph{Finite Resolution Recurrence Variation} of $f$ on $\cN$ is defined as the sum
\begin{equation}
\label{FinVar}
\FRRV(f,\cN) = \sum_{K \in I(\cN)} \left|
\Delta_{1,\ldots,n} (f,K)
\right|
\end{equation}
\end{definition}


\FloatBarrier
\subsection{Normalized Finite Resolution Recurrence Variation}
\label{sec:NormVar}

It is clear that the variation grows with the increase in the size of the set, because there are more and more adjacent grid elements along the border with the same difference. Therefore, we propose to normalize the resulting variation by dividing it by a quantity that corresponds to the diameter of the set on which the variation is computed, so that the normalized variation is independent of the resolution at which we investigate the dynamics. In an $n$-dimensional system, we propose to divide by $\sqrt[n]{\card \cN}$.

Moreover, instead of the absolute value of the variation of the recurrence, we are more interested in the relative variation in comparison to the average recurrence times. For example, the recurrence times in the Van der Pol system (see Figure~\ref{fig:recVanderpol}) are considerably larger than in the H\'{e}non system (see Figure~\ref{fig:recHenon}), and thus small relative fluctuations in the recurrence times might yield considerably higher values of the variation in the former map than in the latter. In order to overcome this problem, it seems reasonable to further normalize the variation by dividing it by the mean recurrence time encountered. We thus propose the following.

\begin{definition}[NFRRV]
\label{def:NormFinVar}
Let $\cN \subset \cG(B)$.
Let $f \colon \cN \to \bR$.
Let $\overline{\rec} (\cN)$ denote mean recurrence in $\cN$, that is,
$\sum_{Q \in \cN} \rec (Q) / \card \cN$.
\emph{Normalized Finite Resolution Recurrence Variation (NFRRV)} of $f$ in $\cN$ is the following quantity:
\begin{equation}
\label{NormFinVar}
\NFRRV(f,\cN) := \frac{\FRRV(f,\cN)}{\overline{\rec} (\cN) \sqrt[n]{\card \cN}},
\end{equation}
where $n$ is the dimension of the phase space.
\end{definition}

We show the results of computation of Finite Resolution Recurrence Variation (FRRV) and its normalized version (NFRRV) in Section~\ref{sec:varComp} below.


\FloatBarrier
\subsection{Computation of the Normalized Finite Resolution Recurrence Variation (NFRRV)}
\label{sec:varComp}

Let us begin by considering the Finite Resolution Recurrence Variation (FRRV) introduced in Section~\ref{sec:FinVar}.
In order to check how this quantity changes with the change of the resolution at which a numerical Morse set is computed, we conducted the following experiment. We chose the six specific dynamical systems whose FRR was discussed so far: the three sample systems described in Section~\ref{sec:frr} and the three cases of the Chialvo model shown in Section~\ref{sec:recComp}. For each of these systems, we constructed a numerical Morse set at a few different resolutions. The computed FRRV is shown in Figure~\ref{fig:FinVar} as a function of the size of the set, counted in terms of the number of grid elements.
One can immediately notice the increasing trend in all the cases, which justifies the need for normalization introduced in Section~\ref{sec:NormVar}.

\begin{figure}[htbp]
\centering
\includegraphics[width=0.7\textwidth]{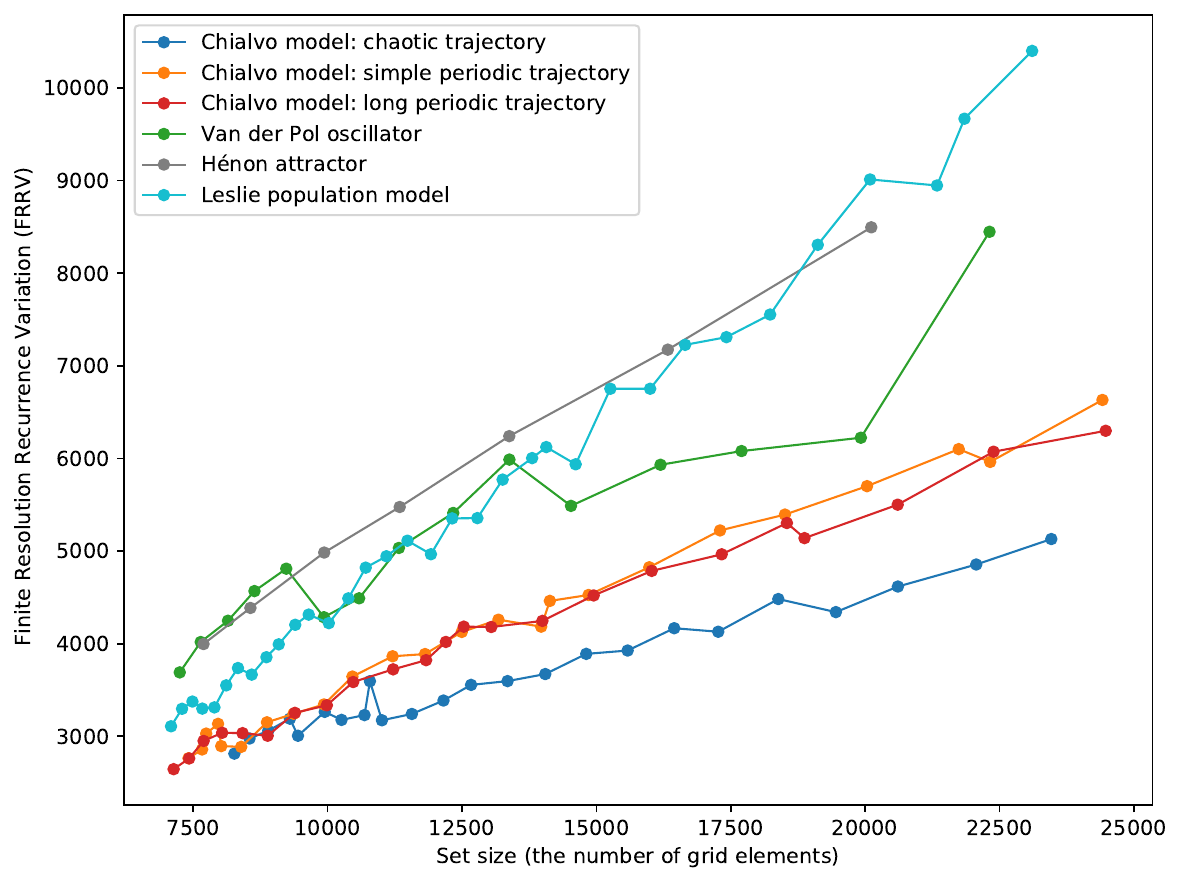}
\caption{\label{fig:FinVar}%
Finite Resolution Recurrence Variation ($\FRRV$) computed for six sample systems at several different resolutions.}
\end{figure}

\begin{figure}[htbp]
\centering
\includegraphics[width=0.7\textwidth]{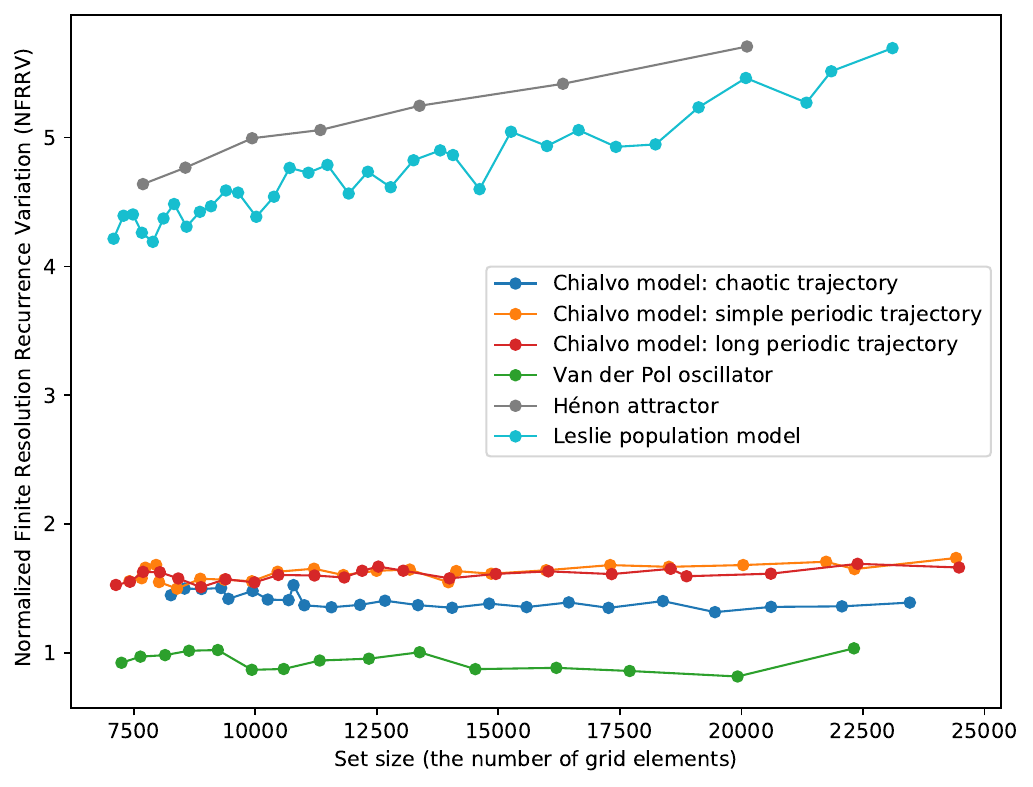}
\caption{\label{fig:NormFinVar}%
Normalized finite resolution recurrence variation ($\NFRRV$) computed for six sample systems at several different resolutions.}
\end{figure}

Figure~\ref{fig:NormFinVar} shows the Normalized Finite Resolution Recurrence Variation (NFRRV) computed for the six sets. One can see that the values computed for the systems that experience complicated dynamics (H\'{e}non, Leslie) still have an increasing trend, which is clearly due to the fact that with the increase in the resolution, more and more details of the dynamics are revealed. However, the values of $\NFRRV$ computed for the Van der Pol oscillator, as well as for the Chialvo model, are essentially constant, with the latter somewhat above the former, which indicates a slightly more complicated dynamics. The fact that the Chialvo model in the parameter regimes corresponding to chaos and long periodic orbits shows smaller values of $\NFRRV$ than the H\'{e}non attractor or the chaotic Leslie model might also be due to the fact that the Chialvo model behaves like a fast-slow system, in which trajectories spend long time in some parts of the phase space and very quickly pass through others, and this behavior is common for almost all trajectories.

A probably counter-intuitive observation is that out of the three sets constructed for the Chialvo model, the lowest values of $\NFRRV$ are encountered by the set with chaotic dynamics. Since the chaotic trajectory wanders round the set in rotational motion, its effect is the actual averaging of recurrence times, which can be seen in Figure~\ref{fig:recNeuron3} as compared to Figures \ref{fig:recNeuron4} and~\ref{fig:recNeuron5}. This is in contrast to the H\'{e}non attractor in which the trajectories tend to run apart each other due to the positive Lyapunov exponent.
\label{rev:slowFast1}
Indeed, the type of chaotic dynamics is substantially different in both cases, because the Chialvo model is a slow-fast system, and the method introduced in \cite{Arai} tends to capture the ``fast'' dynamics only.

\label{rev:frrv}
Let us indicate at this point that computation of FRRV gives rise to method for conducting numerical comparison of dynamics. By comparing the system under investigation with some classical dynamical systems one can obtain certain quantitative information on the complexity of the dynamics. Although FRR values provide rigorous information (e.g. on the lower bound of recurrence time for the actual trajectories, as discussed on page \pageref{frrRigorous} in Section~\ref{frrRigorous}), the FRRV value itself does not provide rigorous evidence of the type of recognized dynamics. On the other hand, FRRV seems to be promising heuristic that enables classification of the dynamics and is in particular useful for the identification of areas of potential chaotic behavior. Thus one of possible applications of FRRV is to precede the application of rigorous numerical proof methods, for which the areas of analysis usually have to be strictly defined.

Figure \ref{fig:varNeuron} shows the results of comprehensive computation of NFRRV for the Chialvo model for a wide range of parameters. Since the computation of FRR is considerably more demanding (in terms of CPU time and memory usage) than the set-oriented analysis of dynamics discussed in Section~\ref{sec:appl}, we conducted the computations for a somewhat narrower range of the parameters $(b,k) \in \Lambda_3 := [0,0.5] \times [0.017,0.027]$ subdivided into $200 \times 50$ boxes, and with the resolution in the phase space reduced to $256 \times 256$. The phase space was taken as $B_3 := [-0.1,7.5] \times [-1.3,2.7]$. Like previously, we fixed $a := 0.89$ and $c := 0.28$. The computation time was a little over $372$ CPU-hours and the memory usage did not exceed $491$~MB per process. The continuation diagram that we obtained was similar to the relevant part of the diagram shown in Figure~\ref{fig:n12cont}, and the Conley-Morse graphs with phase space images can be browsed at~\cite{www}, additionally with recurrence diagrams and the corresponding histograms created for the largest numerical Morse set observed in each computation. Since the value of NFRRV computed for small sets (say, with $\card \cN < 1000$) does not correspond to the intuitive interpretation, the values computed for all the numerical Morse sets of less than $1000$ elements were shown in gray in Figure~\ref{fig:varNeuron}. We use the results of these computations in Section~\ref{sec:recTool}.

\begin{figure}[htbp]
\includegraphics[height=0.77\textwidth]{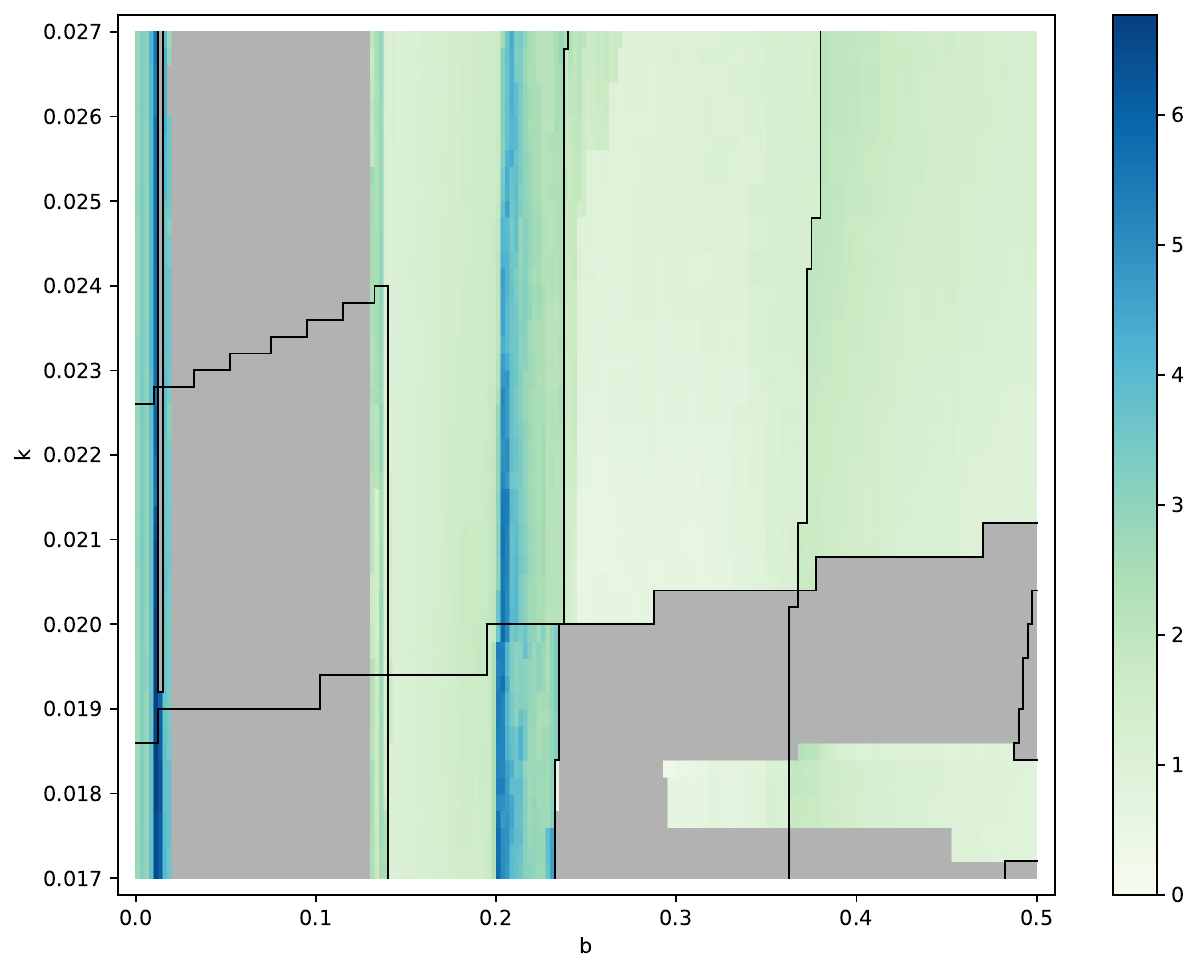}$\qquad$
\caption{\label{fig:varNeuron}%
Normalized Finite Resolution Recurrence Variation ($\NFRRV$) computed for the Chialvo model \eqref{main} with $a = 0.89$, $b \in [0,0.5]$ (horizontal axis), $c = 0.28$, and $k \in [0.017,0.027]$ (vertical axis). The black lines indicate borders between different continuation classes. Values computed for numerical Morse sets of less than $1{,}000$ elements were grayed out.}
\end{figure}


\section{Recurrence as a tool in classification of dynamics}
\label{sec:recTool}

In order to demonstrate the usefulness of the quantities that were introduced in Sections \ref{sec:frr}--\ref{sec:NormVar} (specifically, FRR and NFRRV), we conducted the following two analyses. We used the data computed in Section~\ref{sec:varComp} to group parameters by the shape of the corresponding recurrence histogram (discussed in Section~\ref{sec:clusteringHist}), and by the value of NFRRV together with the median of FRR (discussed in Section~\ref{sec:clusteringFRR}). We used the unsupervised machine learning density-based spatial clustering algorithm DBSCAN that finds prominent clusters in the data and leaves the remaining points as ``noise''. The results of this kind of clustering are shown in Figures \ref{fig:learn25d_h10} and \ref{fig:learn25d_h00}, respectively.

 We would like to point out the fact that our analysis of variation of recurrence times within large invariant sets is aimed at providing some quantification of chaos in the Chialvo model and developing guidelines for how to analyze other systems. So far, to the best of our knowledge, chaos in the Chialvo model has not been studied analytically with the exception of the work \cite{Jing} where the authors studied chaos in the sense of Marotto (existence of the so-called snap-back repeller) which is closer to the notion of Li-Yorke chaos than chaos in the sense of Devaney (see e.g.\ \cite{Zhao2009}). However, the existence of Marotto chaos in their work is mostly relevant to the case $k<0$, thus not much connected with our results. The existence of the snap-back repeller in biological systems was examined in \cite{Liao2018} with the use of computations conducted in interval arithmetic.


\FloatBarrier
\subsection{Recurrence histograms}
\label{sec:clusteringHist}

In our first analysis, we took the recurrence histograms, such as those shown in Figures \ref{fig:recNeuron3}, \ref{fig:recNeuron4}, \ref{fig:recNeuron5}. We removed the bar corresponding to recurrence $1$ from each of the histograms, and we scaled the horizontal axis to make each histogram consist of precisely $10$ bars. We then normalized the histograms so that the sum of the heights of the bars was equal $1$. Both the original and reduced histograms can be browsed at \cite{www}. We equipped the $10$-dimensional space of all such histograms with the $l^1$ metric (also known as the Manhattan metric, or the taxicab metric). We restricted our attention to all the histograms obtained for the largest numerical Morse set at each of the $200 \times 50$ parameter boxes, provided that this set consisted of at least $1{,}000$ boxes in the phase space. Then we ran the DBSCAN clustering algorithm on this collection of histograms. This algorithm finds core samples of high density and expands clusters from them. It takes two parameters: the maximum distance $\varepsilon$ between two samples for one to be considered in the neighborhood of the other, and the number $p$ of samples in a neighborhood of a point to be considered as a core point. We tried all the combinations of $\varepsilon \in \{0.1, 0.2, \ldots, 1.3, 1.4\}$ and $p \in \{50, 100, 150, 200, 250, 300\}$. The most useful clustering was obtained for $\varepsilon = 0.2$ and $p = 150$, and is shown in Figure~\ref{fig:learn25d_h10}.

\begin{figure}[htbp]
\centering
\includegraphics[height=0.7\textwidth]{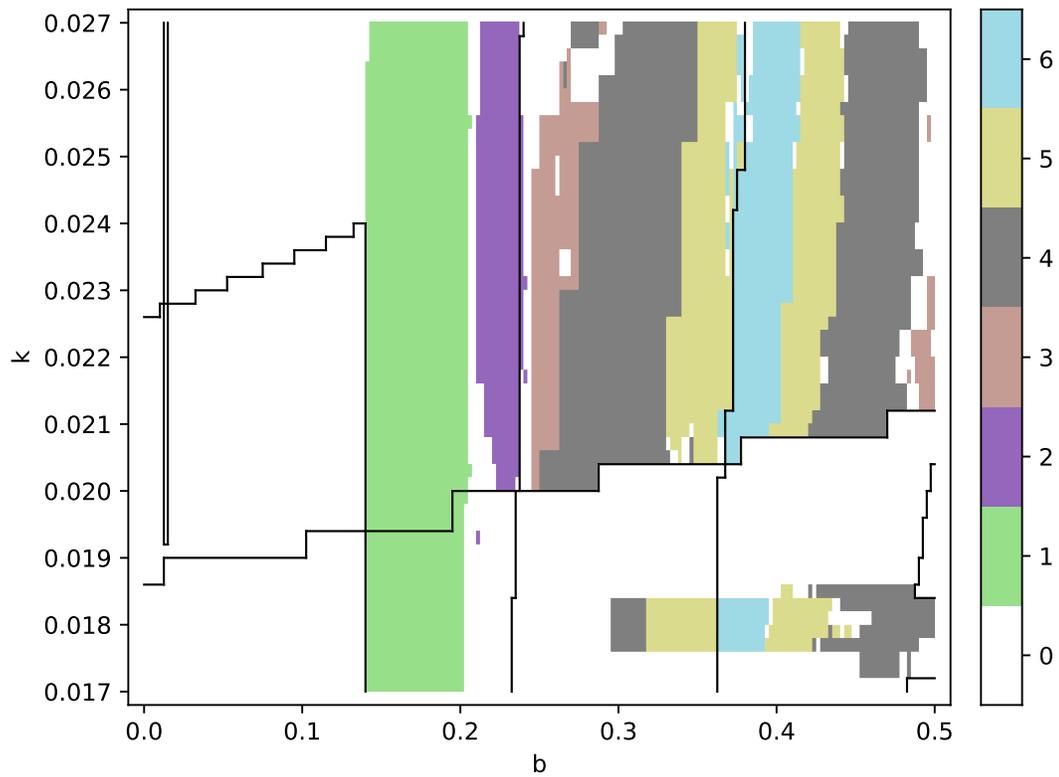}$\qquad$
\caption{\label{fig:learn25d_h10}%
Clusters of similar recurrence histograms found by the DBSCAN algorithm in the Chialvo model \eqref{main} with $a = 0.89$, $b \in [0,0.5]$ (horizontal axis), $c = 0.28$, and $k \in [0.017,0.027]$ (vertical axis). The black lines indicate borders between different continuation classes. Histograms for numerical Morse sets of less than $1{,}000$ elements were not taken into consideration.}
\end{figure}

It is interesting to see the histograms that correspond to the consecutive classes; they are shown in Figure~\ref{fig:learn25d_hist}. The first histogram (Class~1) shows the presence of many boxes in the numerical Morse set with a wide range of recurrence times. The shape of the histogram resembles normal distribution skewed to the right and thus is most similar to the histogram obtained for the H\'{e}non map (see Figure~\ref{fig:recHenon}). This shape of a histogram might thus be an indicator of chaotic dynamics. Indeed, this region of parameters coincides with the neighborhood of part of the region in which large attractors were found in numerical simulations (see Figure~\ref{fig:n25dsize} in Appendix~\ref{app:simulations}) that also indicate likely chaotic dynamics.

The second histogram (Class~2) has three major peaks: the highest one in the middle, another at high recurrence, and a considerably less prominent one for some lower recurrence values. This result reflects the internal structure of the numerical Morse set that can be seen in Figure~\ref{fig:recNeuron25d_90_31}, and both the histogram and the recurrence diagram resemble the situation observed in the Leslie system shown in Figure~\ref{fig:recLeslie}.

The next four histograms are similar to each other, with the high peak around $34$, but with higher and higher recurrence values appearing and thus pushing the peak to the left (it gets shifted from the 3rd position from the right to the 4th, 5th, and 6th positions, respectively). All the cases are somewhat similar to the Van der Pol oscillator case shown in Figure~\ref{fig:recVanderpol}, and might indicate simple, non-chaotic dynamics, at least as perceived at the finite resolution.

\begin{figure}[htbp]
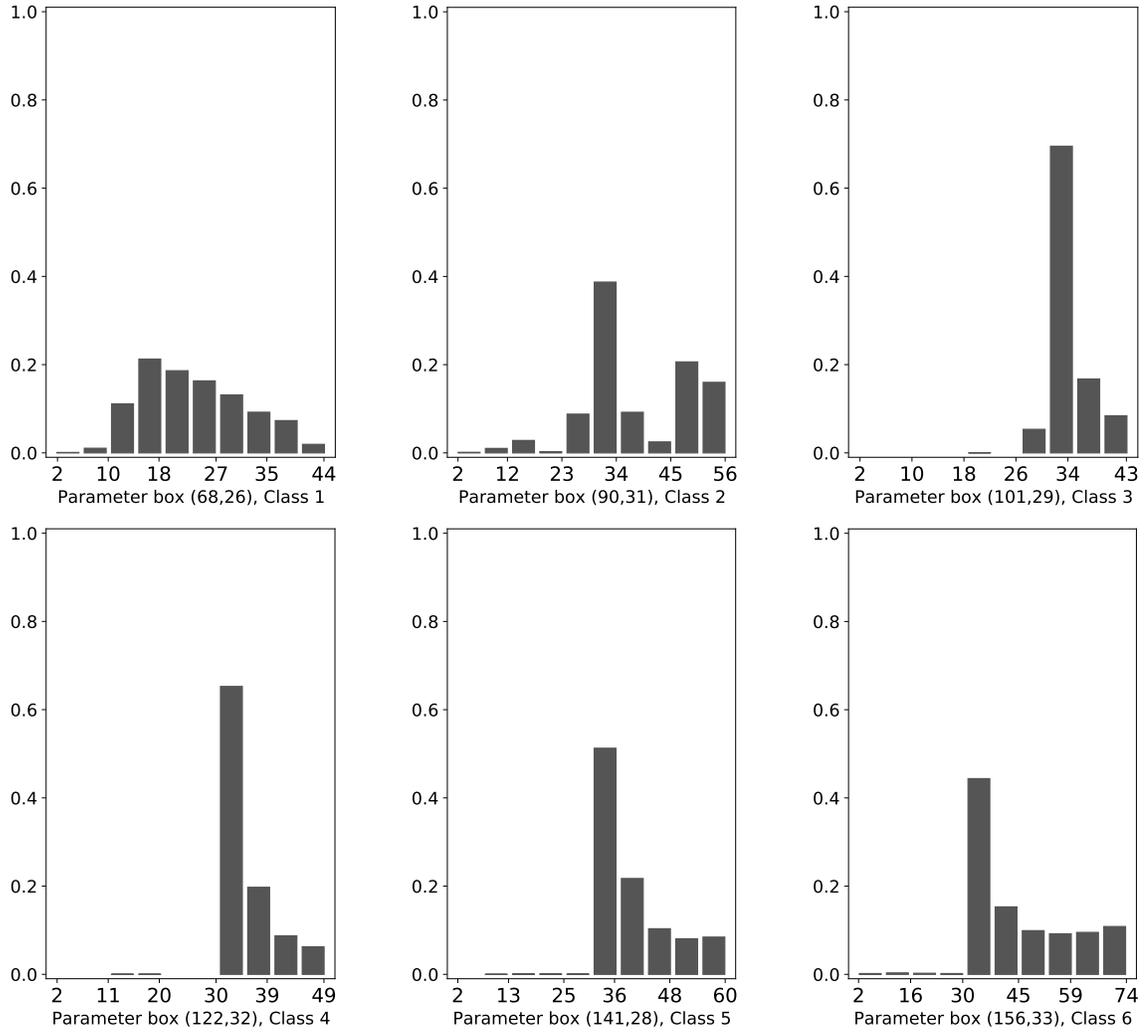

\centerline{\includegraphics[width=0.3\textwidth]{\detokenize{learn25d_68_26}} \hfill
\includegraphics[width=0.3\textwidth]{\detokenize{learn25d_90_31}} \hfill
\includegraphics[width=0.3\textwidth]{\detokenize{learn25d_101_29}}}
\centerline{\includegraphics[width=0.3\textwidth]{\detokenize{learn25d_122_32}} \hfill
\includegraphics[width=0.3\textwidth]{\detokenize{learn25d_141_28}} \hfill
\includegraphics[width=0.3\textwidth]{\detokenize{learn25d_156_33}}}
\caption{\label{fig:learn25d_hist} Reduced recurrence histograms computed for representative parameter boxes from the classes shown in Figure~\ref{fig:learn25d_h10}.}
\end{figure}

\begin{figure}[htbp]
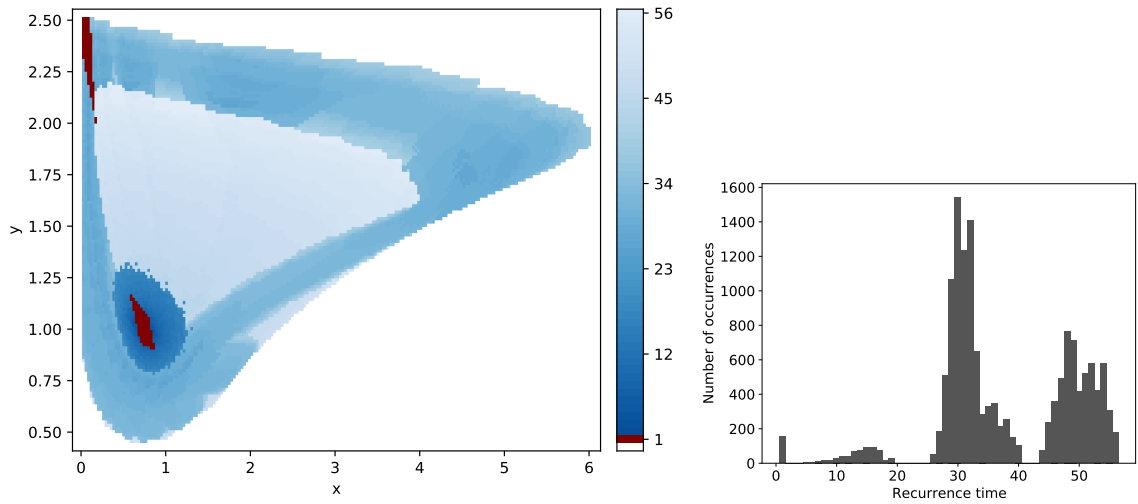

\centering
\includegraphics[width=0.6\textwidth]{\detokenize{recNeuron25d_90_31}}\hfill
\includegraphics[width=0.39\textwidth]{\detokenize{recNeuron25d_90_31Hist}}
\caption{\label{fig:recNeuron25d_90_31}%
Recurrence diagram with histogram computed for the large numerical Morse set found for the Chialvo map with $a = 0.89$, $b \in [0.225,0.2275]$, $c=0.28$, and $k\in [0.0232,0.0234]$. These parameters correspond to the box no.\ $(90,31)$ of parameters in the computation described in Section~\ref{sec:varComp}.}
\end{figure}


\FloatBarrier
\subsection{Clustering based on FRR and NFRRV}
\label{sec:clusteringFRR}

In our second analysis, we took the pairs of two numbers: NFRRV and the median of FRRs computed for the largest numerical Morse set at each of the $200 \times 50$ parameter boxes. We only considered numerical Morse sets of at least $1{,}000$ boxes in the phase space. We standardized each of the two variables (subtracted the mean and divided by the standard deviation). We tried DBSCAN with the same values of $\varepsilon$ and $p$ as in the first analysis, but we first multiplied the variables by $7$ in order to compensate for the lower dimension of the space. We used the $l^1$ metric. The most reasonable clustering was obtained for $\varepsilon = 0.8$ and $p = 100$, and is shown in Figure~\ref{fig:learn25d_h00}.

\begin{figure}[htbp]
\includegraphics[height=0.7\textwidth]{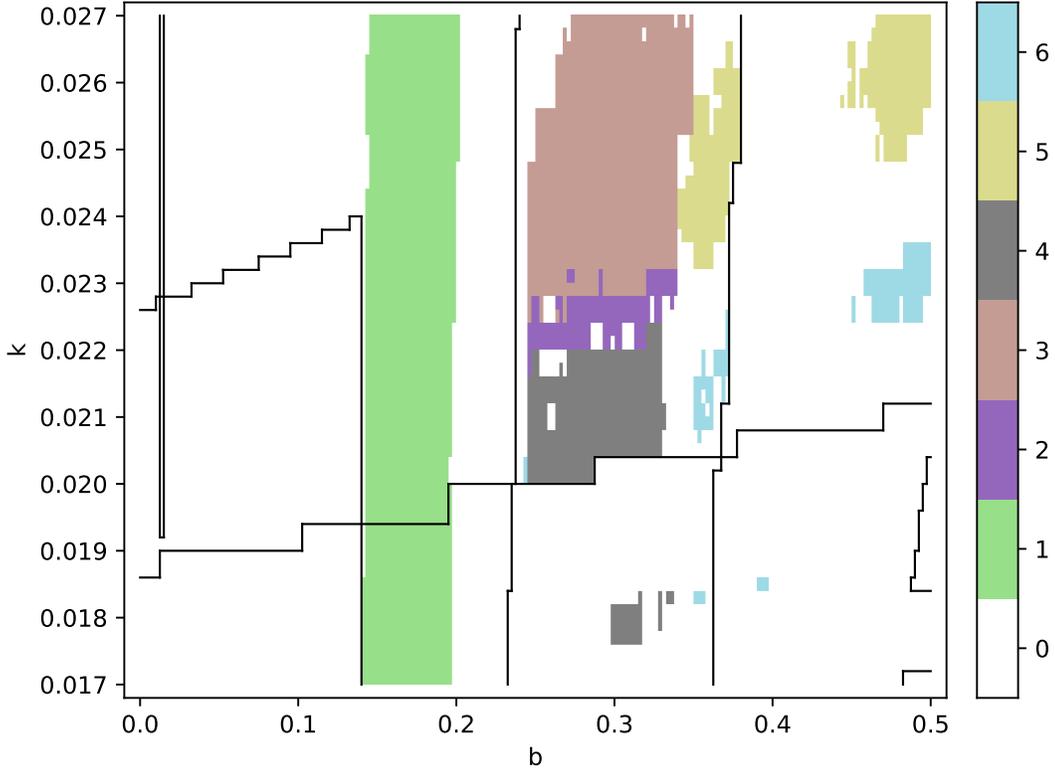}$\qquad$
\caption{\label{fig:learn25d_h00}%
Clusters of similar values of NFRRV together with the median of FRR found by the DBSCAN algorithm in the Chialvo model \eqref{main} with $a = 0.89$, $b \in [0,0.5]$ (horizontal axis), $c = 0.28$, and $k \in [0.017,0.027]$ (vertical axis). The black lines indicate borders between different continuation classes. Histograms for numerical Morse sets of less than $1{,}000$ elements were not taken into consideration.}
\end{figure}

The characteristics of each of the six clusters found in this computation are shown in Table~\ref{tab:clusters}. One can notice that the features that distinguish Cluster~1 from the others is its lowest FRR together with relatively high NFRRV, which might be an indicator of chaotic dynamics. Indeed, low values of FRR were found in the H\'{e}non system (see Figure~\ref{fig:recHenon}) and in the Leslie model (see Figure~\ref{fig:recLeslie}), with very high values of NFRRV (see Figure~\ref{fig:NormFinVar}). Like in the first analysis, this region of parameters coincides with the neighborhood of part of the region in which large attractors were found in numerical simulations (see Figure~\ref{fig:n25dsize} in Appendix~\ref{app:simulations}) that also indicate likely chaotic dynamics.

Cluster~3 has the second lowest median FRR; however, its NFRRV values are relatively low. Even though numerical simulations found some relatively large attractors in some of the parameters in Cluster~3 (see Figure~\ref{fig:n25dsize} in Appendix~\ref{app:simulations}), the FRR and NFRRV values do not suggest the existence of chaotic dynamics and comply with findings of our first analysis.

An interesting observation is that the considerably higher values of NFRRV in Clusters 5 and~6 provide means to distinguish the different types of dynamics found in the different parts of Clusters 4 and~5 found in the first analysis (see Figure~\ref{fig:learn25d_h10}).

Clusters 2 and~4 have considerably lower values of NFRRV and higher median FRR than were found in the other clusters. Their location is just below the possible chaotic dynamics found in the numerical simulations (see Figure~\ref{fig:n25dsize} in Appendix~\ref{app:simulations}) and thus the dynamics might be more similar to what was found in the Van der Pol oscillator (see Figure~\ref{fig:recVanderpol}).

\begin{table}[htb]
\begin{tabular}{cccc}
\hline
cluster & size & NFRRV & median FRR \\
\hline
1 & 1133 & $1.378 \pm 0.215$ & $22.675 \pm 0.648$ \\
2 & 123 & $0.753 \pm 0.046$ & $35.203 \pm 0.584$ \\
3 & 752 & $0.927 \pm 0.080$ & $31.134 \pm 1.479$ \\
4 & 327 & $0.635 \pm 0.046$ & $38.780 \pm 1.384$ \\
5 & 315 & $1.287 \pm 0.067$ & $34.032 \pm 1.060$ \\
6 & 116 & $1.084 \pm 0.045$ & $40.000 \pm 0.643$ \\
\hline
\end{tabular}
\vskip 6pt
\caption{\label{tab:clusters} Characteristics of the clusters found using the following two variables: NFRRV and median FRR. The size is given in the number of parameter boxes, the other values are provided as mean $\pm$ standard deviation.}
\end{table}


\section{Final remarks}
\label{sec:conclusion}

To sum up, the method that we propose in this paper allows one to obtain an overview of global dynamics in a given system, to arrange the set of parameters depending on the type and complexity of dynamics they yield, and to compare the observed dynamics to other known systems. In particular, it is possible to identify ranges of parameters in which complicated dynamics is likely to occur. All this can be done through automated computation that requires little human effort.

We would like to emphasize the fact that the mathematical framework and the software introduced in this paper are not limited to dimension two. However, we showed an application of this method to a two-dimensional system with two varying parameters for the sake of the ease of visualization of the results. An important objective for further work might thus be to understand the results of the computations without the need for visualizing all the details. The idea of Conley-Morse graphs meets this goal as far as understanding the global dynamics in the phase space is considered. However, getting hold of the interplay between changes in the various parameters of the system and the corresponding changes in the dynamics might be a more demanding task.

By applying this method to the analysis of the Chialvo model, we obtained a comprehensive overview of its global dynamics across a wide range of parameters. We also used a heuristically justified method based on the computation of recurrence to distinguish between the occurrence of periodic and chaotic dynamics.
\label{rev:slowFast2}
Note that the method introduced in \cite{Arai} does not in principle apply well to fast-slow systems, because it generally captures the ``fast'' dynamics only due to its nature. However, by adding Finite Resolution Recurrence analysis combined with machine learning, we were able to somewhat distinguish between different kinds of ``slow'' dynamics. We believe that it is worth to conduct further research in this direction.

The challenge for the future research would be to provide a complete mosaic of bifurcation patterns. As it was illustrated in Figure \ref{fig:traj4}, this would be a demanding task, and thus we may repeat after Chialvo \cite{Chialvo} that more ``work is still needed to fully understand the bifurcation structure of Equation~\eqref{main}.''


\section*{Acknowledgements}
This research was supported by the National Science Centre, Poland, within the following grants:
Sheng~1 2018/30/Q/ST1/00228 (for G.G.),
SONATA 2019/35/D/ST1/02253 (for J.S-R.),
and OPUS 2021/41/B/ST1/00405 (for P.P.).
J. S-R. also acknowledges the support of Dioscuri program initiated by the Max Planck Society, jointly managed with the National Science Centre (Poland), and mutually funded by the Polish Ministry of Science and Higher Education and the German Federal Ministry of Education and Research.
Computations were carried out at the Centre of Informatics Tricity Academic Supercomputer \& Network.
The authors would like to express their gratitude to the anonymous reviewers for their constructive remarks that helped improve the quality of the paper.


\section*{Authors' Declarations}

\subsection*{Conflict of Interest}

The authors have no conflicts to disclose.

\subsection*{Author Contributions}

\textbf{Pawe\l \ Pilarczyk:} Conceptualization (equal);
Data curation (lead);
Formal analysis (equal);
Funding acquisition  (equal);
Investigation (equal);
Methodology (equal);
Software (lead);
Visualization (equal);
Validation (equal);
Writing – original draft (equal);
Writing – review \& editing  (equal). 

\textbf{Justyna Signerska-Rynkowska:} Conceptualization (equal);
Formal analysis (equal);
Funding acquisition  (equal);
Investigation (equal);
Methodology (equal);
Visualization (equal);
Validation (equal);
Writing – original draft (equal);
Writing – review \& editing  (equal).   

\textbf{Grzegorz Graff:} Conceptualization (equal);
Formal analysis (equal);
Funding acquisition  (equal);
Investigation (equal);
Methodology (equal);
Validation (equal);
Writing – original draft (equal);
Writing – review \& editing  (equal). 

\subsection*{Data Availability} 
The data upon which our results are based is publicly and freely available at~\cite{DATA}, and an interactive browser of the data together with the source code of the software is available at~\cite{www}.


\pagebreak


\appendix

\section{Algorithmic computation of recurrence times}
\label{app:recurrence}

This appendix contains an effective algorithm that we propose for the computation of Finite Resolution Recurrence (FRR) introduced in Section~\ref{sec:frr}, proof of its correctness, and analysis of its computational complexity.

If $\cN$ is a numerical Morse set then one can restrict the codomain of the multivalued map $\cF$ on $\cG(B)$ to $\cN$ in order to obtain the multivalued map $\cF \colon \cN \multimap \cN$ (with some images possibly empty).
Let $G = (V,E)$ be the directed graph that represents this multivalued map. The recurrence time of $Q \in \cN$ is the length of the shortest non-trivial cycle passing through $Q$. If $\cN$ is a strongly connected path component of the graph $G$ then $\rec (Q)$ is finite.
The following algorithm allows one to compute the recurrence times for all the elements of $\cN$ effectively:

\begin{algorithm}[computation of recurrence times]\noindent\label{alg:rec}
\rm
\begin{tabbing}
\hspace{.5 cm}\= \hspace{.5 cm}\= \hspace{.5 cm}\= \hspace{.5 cm}\= %
\hspace{.5 cm}\= \hspace{.5 cm}\= \hspace{.5 cm}\= \hspace{.5 cm}\= \kill
\kw{function} recurrence\textunderscore{}times \\
\kw{input:} \\
\> $G = (V,E)$: a directed graph \\
\kw{begin} \\
\> \kw{for each} $v \in V$: \\
\>\> \kw{for each} $u \in V$: \\
\>\>\> $D_{v u} := $ length of the shortest path in $G$ from $v$ to $u$ \\
\> \kw{for each} $v \in V$: \\
\>\> \kw{if} $(v,v) \in E$ \kw{then} \\
\>\>\> $\rec(v) := 1$ \\
\>\> \kw{else} \\
\>\>\> $\rec(v) := \min \{D_{v u} + D_{u v} : u \in V\}$ \\
\kw{end.}
\end{tabbing}
\end{algorithm}

In this algorithm, one first computes the shortest paths from every vertex $v \in V$ to every other vertex $u \in V$. Then two possibilities are considered. If there exists a self-loop from $v \in V$ to itself then $1$ is recorded as the recurrence time for $v$. Otherwise, all the other loops go through some other vertex $u \in V$, and the shortest possibilities are explored to determine an optimal non-trivial loop from $v$ back to itself.

\begin{theorem}[correctness and complexity of computation of recurrence times]
\label{thm:rec}
Algorithm~\ref{alg:rec} called with a directed graph $G = (V,E)$ computes the recurrence times for all the vertices $v \in V$. It can reach this goal within as little as $O(|V||E| + |V|^2 \log |V|)$ time and using $O(|V|^2)$ memory.
\end{theorem}

\begin{proof}
Let us first notice that Algorithm~\ref{alg:rec} computes the recurrence times for all the vertices. Indeed, a cycle through a vertex $v \in V$ is either a self-loop (of length 1), or a cycle that runs through another vertex $u \in V$. Then the algorithm clearly computes the minimum length of any nontrivial cycle that runs through $v$.

Let us now analyze the computational complexity of Algorithm~\ref{alg:rec}. Dijkstra's algorithm can be used to compute simultaneously all the lengths of the shortest paths that start at a fixed vertex $v \in V$. Its time complexity is $O(|E| + |V| \log |V|)$ when implemented using the Fibonacci heap. Therefore, the time complexity of the first loop is $O(|V||E| + |V|^2 \log |V|)$, because Dijkstra's algorithm must be run for every vertex $v \in V$ as a source. The time complexity of the second loop is $O(|V|^2)$, because its body is run for every $v \in V$, and the minimum is computed over all $u \in V$.

The memory complexity is $O(|V|^2)$, or even $\Theta(|V|^2)$, because the entire matrix $D$ must be stored after having been computed in the first loop for the purpose of being used in the second loop.
\end{proof}


\section{Details of the automatic analysis of global dynamics in the Chialvo model}
\label{app:continuation}

This appendix gathers the technical details of application of the methods introduced in Sections \ref{sec:glob} and~\ref{sec:indiv} to the Chialvo model of a neuron explained in Section~\ref{sec:model}. The results are discussed in Sections~\ref{sec:appl}--\ref{sec:sizes}.

The computations were carried out at the Centre of Informatics Tricity Academic Supercomputer \& Network (\url{https://task.gda.pl/en/}) on a cluster equipped with Intel\textsuperscript{\textregistered} Xeon\textsuperscript{\textregistered} CPU E5-2670 v3 at 2.30GHz running 64-bit GNU/Linux with kernel 3.10.0. In the paper we report accumulated CPU time in terms of CPU-hours, and we also indicate the maximum amount of memory used by each single process.

All the real numbers that we mention in the paper should be understood as approximations of the actual binary floating-point numbers stored in the computer. For example, if we say that we set a parameter value to $0.1$ then this means that the actual value used in the computations was a binary representation of $0.1$, which is close but not exactly equal $0.1$. Also, if we say that we obtained a value of $0.30$ then this means that the actual number obtained was a binary floating-point number whose approximation within the provided precision was $0.30$, that is, the number was in $[0.295,0.305]$.

An \label{implementation} implementation of the proposed method for the analysis of global dynamics is available at~\cite{www}.
The software can run on a single computer or on a cluster, using the parallelization scheme introduced in~\cite{Pil2010}; in particular, the software is capable of using multiple CPU cores if run on a single machine in an appropriate way.
The software uses the Computer Assisted Proofs in Dynamics (CAPD) library \cite{CAPD,CAPD2021} for rigorous numerics.
Conley indices are computed using the Computational Homology Project (CHomP) software \cite{CHomP} based on \cite{KMM2004} with further improvements \cite{MMP2005,PR2015,PS2008}.

We applied this method to the system \eqref{main} in the following setup.
We fixed the parameters $a := 0.89$ and $c := 0.28$, and we made the two other parameters vary: $(b,k) \in \Lambda_1 := [0,1] \times [0,0.2]$.
Our choice of the parameters and their ranges was based on what Chialvo indicated in his paper \cite{Chialvo}.
We took the $100 \times 100$ uniform rectangular grid in $\Lambda_1$ to compute the continuation diagram.

Determining a bounded rectangular subset $B$ of the phase space that contained the constructed numerical Morse decomposition in its interior consisted of two steps. We first did a series of numerical simulations to make sure that at least most prominent attractors would be captured. For that purpose, we took $30 \times 30 = 900$ uniformly spaced initial conditions from $[0,100] \times [-100,100]$, and we considered $11 \times 11 = 121$ uniformly spaced parameter values from the ranges of interest. For each of the $121$ parameter combinations, we iterated each of the $900$ initial conditions $10{,}000$ times by the map, and then we analyzed the ranges of the variables for the next $100$ iterations. We found out that $x \in [0.000,7.102]$ and $y \in [-0.575,2.545]$. In the second step, we took $B_0 := [-0.1,10] \times [-5,5]$, and we conducted low-resolution test computation with the set-oriented analysis method, using the $256 \times 256$ uniform grid in the phase space (without computation of Conley indices to save time). This computation took about $10$ minutes on a laptop PC. The constructed numerical Morse sets for all the parameter boxes considered were contained in $[-0.022,8.55] \times [-4.54,2.78]$. We thus took $B := [-0.1,9] \times [-5,3]$ for the final computation. We remark that it would be ideal to obtain some bounds for the recurrent dynamics analytically, but the only obvious bound in our case was $x \geq 0$, and the absorbing region $D^+$ determined in \cite[p.~1643]{courbage2010} turned out not to contain all the recurrent dynamics we were interested in and thus could not be used for our purpose.

Once the set $B$ has been fixed, we conducted the computations with the $1024 \times 1024$ uniform rectangular grid in $B$. We explain that due to the gradual refinements scheme \cite[\S 4.2]{Arai}, the phase space must be subdivided into a $d \times \cdots \times d$ uniform rectangular grid, where $d$ is a power of $2$. The computation completed within $29$ CPU-hours, and each process used no more than $1.8$~GB of memory. All the numerical Morse sets fit in $[-0.003,7.24] \times [-1.35,2.58]$.
Between $1$ and $4$ numerical Morse sets were found in each numerical Morse decomposition, typically $1.24\pm 0.5$ (average $\pm$ standard deviation).
Either one or two of them were attracting, typically $1.015 \pm 0.122$.
The sizes of the total of $12{,}378$ Morse sets were between $1$ and $106{,}978$ grid elements, typically $10{,}863 \pm 19{,}065$, with the median of $1{,}193$.
Ten continuation classes were found, including one consisting of a single parameter box.
The resulting continuation diagram is shown in Figure~\ref{fig:n09cont}, and the sizes of the constructed outer bounds for the chain recurrent set are shown in Figure~\ref{fig:sizeNeuron09}.

\begin{figure}[htbp]
\centerline{\includegraphics[width=0.8\textwidth]{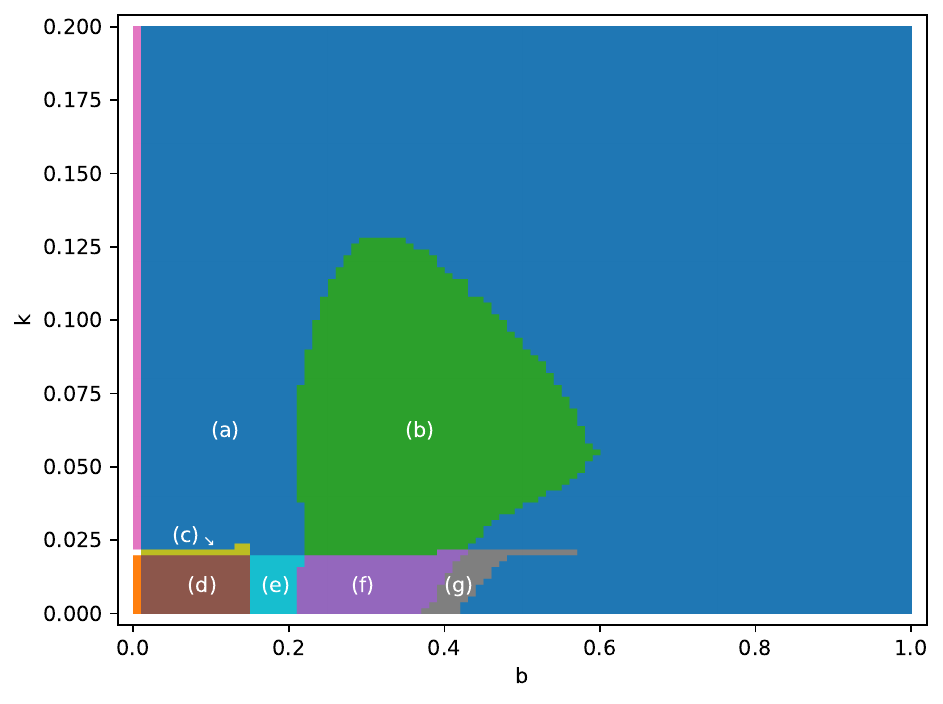}}
\caption{\label{fig:n09cont}%
Continuation diagram for the Chialvo model with $a = 0.89$, $c = 0.28$, and $(b,k) \in \Lambda_1 = [0,1] \times [0,0.2]$ split into the $100 \times 100$ uniform rectangular grid. The areas (a)--(g) correspond to the areas marked with the same labels in Figure~\ref{fig:n12cont}, but note that the colors are different in most cases, because the colors are assigned automatically to the classes by the software.}
\end{figure}

\label{sec:sizeNeuron09}

\begin{figure}[htbp]
\centering
\includegraphics[width=0.9\textwidth]{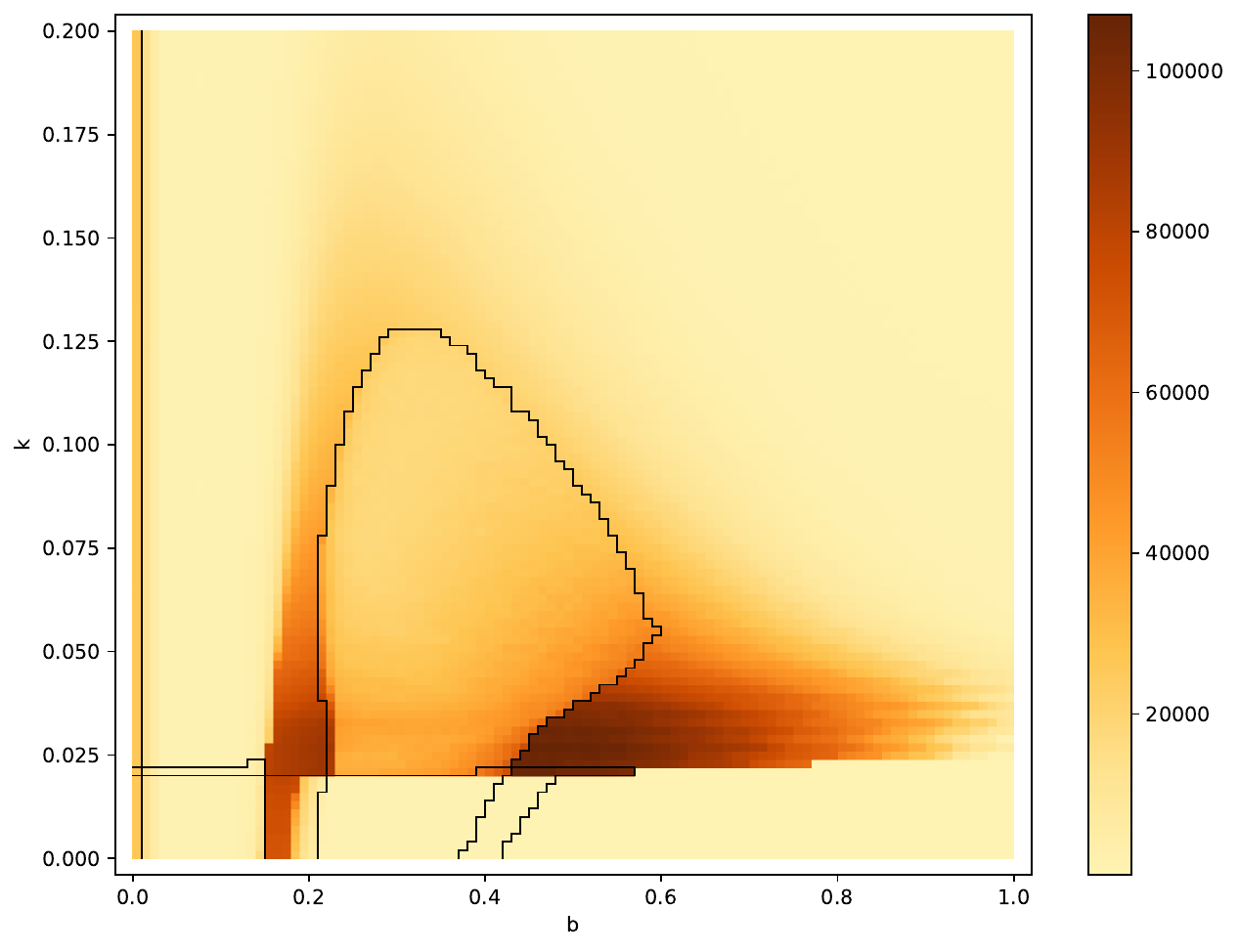}
\caption{\label{fig:sizeNeuron09}%
The size of the union of all the numerical Morse sets found in the phase space for the corresponding parameters in $\Lambda_1$ in the Chialvo model. The black lines indicate borders between different continuation classes; see Figure~\ref{fig:n09cont} for the corresponding continuation diagram.}
\end{figure}

Since most changes in global dynamics appear to take place for a limited range of~$k$, we conducted an analogous computation with the range of parameters limited to \label{Lambda} $(b,k) \in \Lambda_2 := [0,1] \times [0.015,0.030]\subset \Lambda_1$. The set $\Lambda_2$ was split into the $200 \times 75$ uniform rectangular grid. The computation took a little over $34$ CPU-hours and memory usage did not exceed $396$~MB per process. All the numerical Morse sets were contained in $[0.015,7.07] \times [-1.33,2.56]$. The corresponding continuation diagram is shown in Figure~\ref{fig:n12cont}. Note, however, that obviously the colors assigned automatically to the continuation classes differ from those shown in Figure~\ref{fig:n09cont}. Unfortunately, this lack of color match between different continuation diagrams is unavoidable. One of the reasons is that one class in a diagram computed at some scale may split into two or more classes in the diagram computed at a different scale; moreover, two classes may not be related by continuation even if the numbers and stability types of numerical Morse sets are the same, so giving an impression that colors correspond to types of dynamics might be
deceiving.
\label{rev:colorMismatch}


\FloatBarrier
\section{Attractor sizes in numerical simulations}
\label{app:simulations}

In order to find regions of parameters for which chaotic attractors might exist, we conducted the following numerical simulations aimed at detecting large attractors.

We took a set of evenly spaced $400 \times 400$ values of $(b,k) \in \Lambda_3 = [0,0.5] \times [0.017,0.027]$. For each of these parameter values, with $a := 0.89$ and $c := 0.28$ fixed, we computed four trajectories in the Chialvo model \eqref{main}, starting at $(0.01,0.01)$, $(1.0, 1.0)$, $(10.0, 2.0)$, and $(2.0, 10.0)$, respectively. We discarded the first $10{,}000$ iterations of each trajectory in order to make it settle on a hypothetical attractor, and then we recorded the results of the next $10{,}000$ iterations. In this way, for each of the $400 \times 400$ values of the parameters $(b_i,k_j) \in \Lambda_3$, we obtained four sequences of points indexed by $l=1,\ldots,4$. Let $A_{i,j,l}$ denote the union of the points in each sequence. Each of the sets $A_{i,j,l}$ serves as a numerical approximation of some attractor found in the system.

We were interested in measuring the size of each attractor that we found, in hope for numerically detecting some chaotic attractors. For that purpose, we applied an approach similar to one step of the computation of Hausdorff dimension of the set. Namely, we considered the uniform rectangular grid $\cG$ at the resolution $4 \times 4$ times finer than in the computation upon which Figure~\ref{fig:varNeuron} was based. For each set $A_{i,j,l}$, we constructed the collection $\cN_{i,j,l}$ of those boxes in $\cG$ that contain at least one point from $A_{i,j,l}$. The set $\cN_{i,j,l}$ could be understood as the minimal covering of $A_{i,j,l}$ with respect to $\cG$, or a plot of $A_{i,j,l}$ at the resolution defined by $\cG$. We then computed $C_{i,j,l} := \card \cN_{i,j,l}$. In Figure~\ref{fig:n25dsize}, we show the value of $C_{i,j} := \max \{C_{i,j,l} : l = 1, \ldots, 4\}$ for each parameter selection $(b_i,k_j)$ separately.

Let us now refer to Figure~\ref{fig:traj4} with the four types of orbits present in the Chialvo model. Note that computing the diameter of the constructed attractor alone would not be very helpful in distinguishing between these types, because a periodic spiking oscillation would yield similar values as a chaotic attractor. This is why we computed the cover $\cN_{i,j,l}$ instead. Indeed, the seemingly chaotic orbit shown in Figure~\ref{fig:traj4} in blue fills a large area in the phase space, and thus would yield a large value of $C_{i,j,l}$. In contrast to this, the long periodic orbit shown in that figure in red would yield a moderate value of $C_{i,j,l}$.

To sum up, we consider it justified to use the results of these simulations in Section~\ref{sec:recTool} as a reference to the regions of parameters where chaotic dynamics might be present. Although this is not strictly proof, one can call it strong numerical evidence.

\begin{figure}[htbp]
\includegraphics[height=0.77\textwidth]{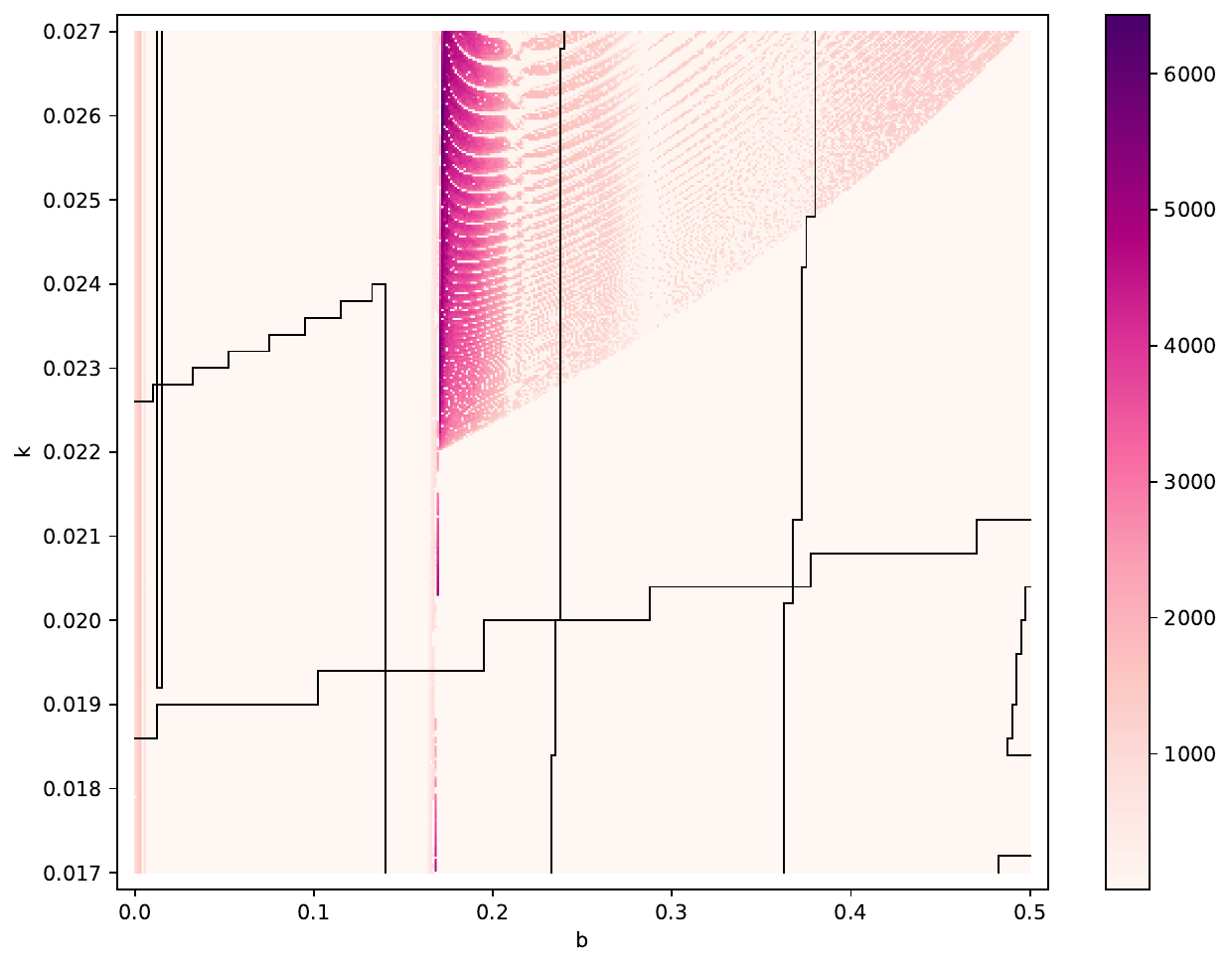}
\caption{\label{fig:n25dsize}%
Maximal size of a cover of an attractor found in numerical simulations for the Chialvo model \eqref{main} with $a = 0.89$, $b \in [0,0.5]$ (horizontal axis, $400$ individual values), $c = 0.28$, and $k \in [0.017,0.027]$ (vertical axis, $400$ individual values). The black lines indicate borders between different continuation classes.}
\end{figure}

\end{document}